\documentclass[12pt, reqno,amsmath,amsthm,amssymb,amscd]{amsart}
\usepackage{pifont}
\usepackage{amssymb}
\usepackage{mathrsfs,amssymb}
\usepackage{multicol}\multicolsep=0pt
\usepackage[enableskew,vcentermath]{youngtab}

\usepackage[sort]{cite}

\def\crulefill{\leavevmode\leaders\hrule height 1pt\hfill\kern 0pt}
\long\def\QUERY#1{%
\leavevmode\newline%
\noindent$\star\star\star$\thinspace\textsf{Comment/Query}\crulefill\newline%
   \space #1\newline\hbox to 120mm{\crulefill}$\star\star\star$\newline}
\newtheorem{Theorem}{Theorem}[section]
\newtheorem{Lemma}[Theorem]{Lemma}
\newtheorem{Cor}[Theorem]{Corollary}
\newtheorem{Prop}[Theorem]{Proposition}

 \theoremstyle{definition}
\newtheorem{example}[Theorem]{Example}

\newtheorem{Defn}[Theorem]{Definition}

\newtheorem{rem}[Theorem]{Remark}

\numberwithin{equation}{section}
\theoremstyle{definition}

\makeatletter
\def\enumerate{\begingroup\ifnum\@enumdepth>3\@toodeep\else
      \advance\@enumdepth\@ne
      \edef\@enumctr{enum\romannumeral\the\@enumdepth}%
      \topsep\z@\parskip\z@
      \list{\csname label\@enumctr\endcsname}
        {\@nmbrlisttrue\let\@listctr\@enumctr
         \parsep\z@\itemsep\z@\topsep\z@
         \setcounter{\@enumctr}{0}
         \def\makelabel##1{\hss\llap{\rm ##1}}
       }\fi}

\makeatother

\let\bar=\overline
\let\epsilon=\varepsilon
\def\({\big(}
\def\){\big)}

\def\0{\underline{0}}

\def\B{\mathscr B}

\def\G{\mathcal G}
\def\H{\mathscr H}

\DeclareMathOperator{\Rad}{Rad} \DeclareMathOperator*{\Res}{Res}

\def\m{\mathfrak m}

\def\Std{\mathscr{T}^{std}}

\def\m{\mathfrak m}
\def\n{\mathfrak n}
\def\s{\mathfrak s}
\def\ts{\tilde\s}
\def\t{\mathfrak t}
\def\u{\mathfrak u}
\def\v{\mathfrak v}

\def\F{\mathcal F}
\def\G{\mathcal G}
\def\Hom{\text{Hom}}

\def\Res{\text{Res}}
\def\U{\mathbf U}

\def\ts{\mathsf t}

\def\res{\text{res} }

{\catcode`\|=\active
  \gdef\set#1{\mathinner{\lbrace\,{\mathcode`\|"8000%
                                   \let|\midvert #1}\,\rbrace}}
  \gdef\seT#1{\mathinner{\Big\lbrace\,{\mathcode`\|"8000%
                                   \let|\midverT #1}\,\Big\rbrace}}
}
\def\midvert{\egroup\mid\bgroup}
\def\midverT{\egroup\,\Big|\,\bgroup}

\def\Set[#1]#2|#3|{\Big\{\ #2\ \Big| \
           \vcenter{\hsize #1mm\centering #3}\Big\}}

\begin{document}
\baselineskip15pt
\title{ Decomposition numbers of  quantized walled Brauer algebras}
\author{ Hebing Rui and Linliang Song }
\address{H.R. Department of Mathematics,  East China Normal
University, Shanghai, 200062, China} \email{hbrui@math.ecnu.edu.cn}
\address{L.S. Department of Mathematics,  East China Normal
University, Shanghai, 200062, China}\email{52110601013@student.ecnu.edu.cn}
\thanks{H. Rui was supported partially  by NSFC in China, Shanghai Municipal Science and Technology Commission
~11XD1402200.}


\begin{abstract} In this paper, we establish explicit relationship between decomposition numbers of quantized walled Brauer algebras and those for
either Hecke algebras associated to certain symmetric groups or (rational) $q$-Schur algebras over a field $\kappa$. This enables us to use   Ariki's result~\cite{Ar} and Varagnolo-Vasserot's result~\cite{VV} to compute
such decomposition numbers
 via inverse Kazhdan-Lusztig polynomials associated with affine Weyl groups of type $A$ if the ground field is $\mathbb C$.
\end{abstract}
\sloppy \maketitle

\section{Introduction}

The quantized  walled Brauer algebra $\mathscr B_{r, s}$ with single parameter
was introduced by   Kosuda and Murakami~\cite{KM} in order to study  mixed
tensor products of natural module and its dual over  quantum general linear  group
$\U_q(\mathfrak{gl}_n)$ over $\mathbb C$. In \cite{Le}, Leduc introduced  quantized walled Brauer algebras $\mathscr B_{r, s}$  with two parameters $\rho$ and $q$.
They are associative algebras over a commutative ring $R$ containing $1$.

It is proved in \cite{Enyang2} that  $\mathscr B_{r, s}$ is  cellular over $R$ in the sense of \cite{GL}. Using standard
results on  representations  of cellular algebras in \cite{GL}, we classified  irreducible $\mathscr B_{r, s}$-modules   over an arbitrary
field  $\kappa$ in \cite{Rsong}. Further, we gave a criterion on the semisimplicity of $\mathscr B_{r, s}$ over $\kappa$. A further question is to compute
 dimensions of irreducible $\mathscr B_{r, s}$-modules in non semisimple case. This can be solved in theory by   determining the multiplicity of any irreducible module in a cell (or standard) module of
$\mathscr B_{r, s}$. Such a multiplicity is called a decomposition number of $\mathscr B_{r, s}$.

The aim of this paper is to compute  decomposition numbers of $\mathscr B_{r, s}$ over $\mathbb C$.
Recently,  various authors have used a variety of techniques to determine decomposition
numbers of Brauer-type algebras. Predominantly this has been via internal considerations~\cite{CV, CV1, Mar, XX}.  But the current
paper is more in the spirit of Donkin-Tange~\cite{DT}, in that it relates these numbers to a Hecke or quantum group
setting via Schur-Weyl duality.

By our result on the semisimplicity of quantized walled Brauer algebras in \cite{Rsong}, we need to compute  decomposition numbers of
$\mathscr B_{r, s}$ under the assumptions either $\rho^2\in q^{2\mathbb Z}$ or not.  In the first case, we classify singular vectors of mixed tensor
product of natural module and its dual over $\U_q(\mathfrak {gl}_n)$ over $\kappa$.  Via the explicit description on such singular vectors, we establish relationships between Weyl modules, partial
tilting modules of  rational $q$-Schur algebras and cell modules,  principle indecomposable modules of $\mathscr B_{r, s}$.  This proves that  decomposition numbers of  $\mathscr B_{r, s}$ can be determined via those for
 (rational) $q$-Schur algebras.  In the second case, we  use Schur functors in \cite[\S4]{Rsong} to set up relationship between
decomposition numbers of $\mathscr B_{r, s}$ and those for Hecke algebras associated to symmetric groups. This enables us to  use Ariki and Varagnolo-Vasserot's results~\cite{Ar, VV} to
compute decomposition numbers of $\mathscr B_{r, s}$  via
the values of inverse Kazhdan-Lusztig  polynomials at $q=1$ when  the ground field is $\mathbb C$.  As a by-product, we  give some partial results on blocks of $\mathscr B_{r, s}$ over a field $\kappa$.

We organize our paper as follows. In section~2, we recall the
definition of $\mathscr B_{r,s}$ and give some  of its properties from \cite{Rsong}.  In section~3, we establish explicit relationship between the decomposition numbers of
$\mathscr B_{r, s}$ over $\kappa$ and those for Hecke algebras associated to certain symmetric groups
  under the assumption that $\rho^2\not\in q^{2\mathbb Z}$.
In section~4, we classify singular vectors of mixed tensor product of natural module and its dual over quantum general linear group $\U_q(\mathfrak {gl}_n)$ over
a field $\kappa$. Via such results, we  set up relationship between decomposition numbers of $\mathscr B_{r, s}$ over $\kappa$ with  $\rho^2\in q^{2\mathbb Z}$ and those for (rational) $q$-Schur algebras  in section~5.
When the ground field is $\mathbb C$, by using Ariki~\cite{Ar}, Varagnolo-Vasserot's results~\cite{VV} on the decomposition numbers for Hecke algebras and $q$-Schur algebras, we obtain the decomposition numbers of $\mathscr B_{r, s}$ no matter whether $q$ is a root of unity or not. By the way, we will also  give some partial results on  blocks of
$\mathscr B_{r,s}$ over $\kappa$.

\section{The quantized walled Brauer  algebra}
In this section, we recall the definition of quantized walled Brauer algebras and state some of its properties from \cite{Rsong}.

Let $\mathcal Z=\mathbb Z[q, q^{-1}]$ be  the ring of Laurent polynomials
in indeterminate $q$.     The Hecke algebra  $\mathscr H_r$  associated to symmetric group $\mathfrak S_r$  is an associative algebra over $\mathcal Z $,  with generators
$g_1, g_2, \cdots, g_{r-1}$ subject to the defining relations: $(g_i-q) (g_i+q^{-1})=0$, $1\le i\le r-1$, and  $g_ig_j=g_jg_i$, if  $|i-j|>1$, and
$g_ig_{i+1}g_i=g_{i+1}g_ig_{i+1}$, $1\le i<r-1$.

Let  $R$ be the localization of $\mathbb Z[q, q^{-1}, \rho,
\rho^{-1}]$ at $q-q^{-1}$, and  let
 $\delta=(\rho-\rho^{-1})(q-q^{-1})^{-1}\in R$.
Fix $r,s \in \mathbb Z^{>0}$. The \textsf{quantized walled Brauer algebra}
${\mathscr{B}}_{r,s}$~\cite{Le} is the associative $R$-algebra with generators
$e_1, g_i, g_j^*$,  $1\le i\le r-1$ and  $1\le j\le s-1$
such that $g_i$'s  are generators of  $\mathscr H_r$ and  $ g_j^*$'s  are generators of  $\mathscr H_s$. Further, the following equalities hold if they make sense:

\begin{multicols}{2}
\begin{enumerate}
    \item $g_ie_1=e_1g_i$, $g^*_ie_1=e_1g^*_i$, $i\neq 1$,
\item $e_1g_1e_1 =\rho e_1=e_1g^*_1e_1$,
\item $e_1^2=\delta e_1$, \item  $g_i g^*_j=g^*_j g_i$,
\item $e_1{g_{1}}^{-
1}g^*_{1} e_1g_{ 1} = {e_1}{g_{1}}^{- 1}g^*_{1 } {e_1}g^*_{1}$,\item
${g_{1}}{e_1}{g_{ 1}}^{ - 1}g^*_{1} e_1 = g^*_{1} e_1{g_{ 1}}^{ -
1}g^*_{1} {e_1}$.
\end{enumerate}
\end{multicols}

\begin{rem}\label{DipperS} In section~4, we will use Dipper-Doty-Stoll's presentation for $\mathscr B_{r, s}$ so as to use their  result in   \cite{DDS1, DDSII} to
discuss singular vectors of mixed tensor products of quantum general linear groups. In that case, $\rho$ and $q$ in Dipper-Doty-Stoll's presentation is  $q^{-1}$ and $\rho^{-1}$ in the
 current  definition of $\mathscr B_{r, s}$,
 respectively.

   \end{rem}

\begin{Lemma}\cite{Enyang2} \label{inv} There is an $R$-linear
anti-involution $\sigma$ on $\mathscr{B}_{r,s}$ which fixes all
generators  $e_1,  g_i$ and $g_j^*$,  $1\le i\le r-1$ and $1\le j\le
s-1$. \end{Lemma}
It is proved in \cite{Enyang2} that $\mathscr B_{r, s}$ is cellular over $R$ in the sense of \cite{GL}. In particular, the rank of $\mathscr B_{r, s}$ is
$(r+s)!$. For any field $\kappa$ which is an $R$-algebra, let $\mathscr B_{r, s, \kappa}=\mathscr B_{r, s}\otimes_R \kappa$.

Let $e$ be the least positive integer such that $1+q^2+\cdots+q^{2(e-1)}=0$ in $\kappa$. If there is no such a positive integer, i.e., $q^2\in \kappa$ is not a root of unity, we set $e=\infty$. The following result has been proved by the authors in \cite{Rsong}.

\begin{Theorem}\label{semimain}\cite[Theorem~6.10]{Rsong} Suppose  $r, s\in \mathbb Z^{>0}$.  Then $\mathscr B_{r, s, \kappa}$ is (split)
semisimple   if and only if $e>\max\{r, s\}$ and one of
the following conditions holds:
\begin{enumerate} \item   $\rho^2\neq
q^{2a}$ for any $a\in \mathbb Z$ with $|a|\le r+s-2$ if $\delta\neq 0$;
\item    $(r, s)\in \{(1,2), (2,1), (1, 3), (3,
1)\}$ if $\delta=0$.
\end{enumerate}
\end{Theorem}

When  we
classify singular vectors in mixed tensor products of natural module and its dual over quantum general linear groups,
we will need explicit description of the cellular basis of $\mathscr B_{r, s}$ in \cite[Theorem~3.7]{Rsong} as follows.
 We need some preparations before we state it.

A composition $\lambda$ of $n$ with at most  $d$ parts is a sequence
of non--negative integers $\lambda=(\lambda_1,\lambda_2,\dots,
\lambda_d)$ such that $|\lambda|:=\sum_{i=1}^d\lambda_i=n$.  If
$\lambda_i\ge \lambda_{i+1}$,  $ 1\le i\le d-1$, then $\lambda$ is
called a partition of $n$ with at most $d$ parts. Let $\Lambda(d,
n)$ (resp. $\Lambda^+(d, n)$)  be the set of all compositions (resp.
partitions)  of $n$ with at most $d$ parts.  We also use
$\Lambda^+(n)$ to denote the set of all partitions of $n$. It is
known that $\Lambda^+(d, n)$ is a poset with dominant order
$\trianglelefteq$ as a partial order on it. More explicitly, $
\lambda\trianglelefteq \mu$  for $\lambda, \mu\in \Lambda^+(d, n)$
if $ \sum_{j=1}^i \lambda_j\le \sum_{j=1}^i \mu_j$ for all possible
$i\le d$. Write $\lambda\vartriangleleft \mu$ if
$\lambda\trianglelefteq \mu$ and $\lambda\ne \mu$.

Let $\lambda=(\lambda_1, \lambda_2,...)\in \Lambda^+(n)$.  The Young diagram $[\lambda]$ is a
collection of boxes (or nodes) arranged in left-justified rows with $\lambda_i$
boxes in the $i$-th row of $[\lambda]$. We use $(i, j)$ to denote the box $p$ if  $p$ is in $i$-th row and $j$-th column.
A box $(i, \lambda_i)$ (resp.,  $(i, \lambda_i+1)$)  is called a removable (resp., an addable ) node of $\lambda$ (or $[\lambda]$)
 if $\lambda_{i}-1\ge \lambda_{i+1}$ (resp. $\lambda_{i-1}\ge \lambda_{i}+1$).   Let $\mathscr R(\lambda)$ (resp., $\mathscr A(\lambda)$) be the
 set of all removable (resp.,  addable ) boxes of $\lambda$. We  use $\lambda\setminus p$ to denote the partition obtained from $\lambda$ by removing the removable node $p$.
 Similarly, we use $\lambda\cup p$ to denote the partition obtained from $\lambda$ by adding  the addable  node $p$.

 A $\lambda$-tableau $\s$ is
obtained by inserting elements $i$ with $ 1\le i\le d$ into $[\lambda]$.
Let $\mu_i$ be the number of $i$ appearing in $\s$. Then $\mu=(\mu_1, \mu_2, \cdots, \mu_d)\in \Lambda(d, n)$.
In this case, $\s$ is called a $\lambda$-tableau of type $\mu$. If the entries of $\s$ increase strictly down the columns  and weakly
increase along the rows, then $\s$ is called a semistandard $\lambda$-tableau of type $\mu$. If we switch the role between columns and rows
of $\s$, then $\s$ is called a column semistandard $\lambda$-tableau of type $\mu$.  Let $\omega=(1,1, \cdots, 1)\in \Lambda^+(n)$.
 A $\lambda$-tableau $\s$ is said to be standard if and only if it is a semi-standard $\lambda$-tableau of type $\omega$.
  Let $\Std(\lambda)$ be the
set of all standard $\lambda$-tableaux.

Now, we focus on $\lambda$-tableaux $\s$ of type $\omega$. Such tableaux will be called $\lambda$-tableaux.  The symmetric group  $\mathfrak S_n $  acts on
$\s$ by permuting its entries. Let $\t^\lambda$ (resp.
$\t_{\lambda}$) be the $\lambda$-tableau obtained from $[\lambda]$ by adding $1, 2, \cdots, n$ from left to right
along the rows (resp. from top to bottom along the columns). For
example, if $\lambda=(4,3,1)$, then
\begin{equation}\label{tla}
\t^{\lambda}=\young(1234,567,8), \quad \text{ and }
\t_{\lambda}=\young(1468,257,3).\end{equation}
We write $w=d(\s)$ if $\t^\lambda w=\s$. Then  $d(\s)$ is
uniquely determined by $\s$.

 Fix $r$ and $s$ and let
\begin{equation}\label{poset}
 \Lambda_{r,s} = \left\{ (f,\lambda )| \lambda \in \Lambda_{r, s}^f, 0\le f \le \min \{ r,s \} \right\},\end{equation}
 where $\Lambda_{r,s}^f =\Lambda^+(r-f)\times \Lambda^+(s-f)$. So,
 each $\lambda\in \Lambda_{r,s}^f$ is of form $(\lambda^{(1)},
 \lambda^{(2)})$. We say that $(f,\lambda)\unrhd (\ell, \mu)$ if either $f>\ell$ or
$f=\ell$ and $\lambda\unrhd \mu$ in the sense
$\lambda^{(i)}\unrhd\mu^{(i)}$, $i=1,2$.  We write
 $(f,\lambda)\rhd (\ell,\mu)$ if $(f,\lambda)\unrhd (\ell,\mu)$
 and  $(f,\lambda)\neq(\ell,\mu)$.
 Then  $\Lambda_{r,s}$ is a poset.

Given a $\lambda\in \Lambda_{r, s}^f$, we define
$\t^\lambda=(\t^{\lambda^{(1)}}, \t^{\lambda^{(2)}})$ where
$\t^{\lambda^{(1)}}$ and $\t^{\lambda^{(2)}}$  are defined similarly
as (\ref{tla}). The only difference is that we have to use $f+i$
instead of $i$ in (\ref{tla}). Similarly, we have $\t_\lambda$.

\begin{example} Suppose $(r, s)=(2, 7)$, $f=1$ and  $(\lambda^{(1)},\lambda^{(2)}) =((1), (3,2,1))$.  We have
\begin{equation}\label{tlar} \t^{\lambda}=\left (\young(2), \quad \young(234,56,7)\right)
\quad \text{and} \quad \t_{\lambda}=\left (\young(2), \quad \young(257,36,4)\right).\end{equation}
\end{example}
For each $\lambda\in \Lambda^f_{r, s}$, let $\Std(\lambda^{(i)})$ be
the set of standard $\lambda^{(i)}$-tableaux which are obtained from
usual standard tableaux by using $f+j$ instead of $j$.  Let $\Std(\lambda)=\Std(\lambda^{(1)})\times \Std(\lambda^{(2)})$.

 For each partition $\lambda$ of $n$, let $\mathfrak S_\lambda$ be  the Young subgroup of $\mathfrak S_n$
 with respect to $\lambda$. Let  $\n_{\lambda}=\sum_{w\in \mathfrak S_{\lambda} }(-q)^{-\ell(w)}  g_w
$ and let $\m_\lambda= \sum_{w\in \mathfrak S_{\lambda} }q^{\ell(w)}  g_w$. Then $\n_\lambda g_i=-q^{-1} \n_\lambda$ and  $\m_\lambda g_i=q \m_\lambda$, if $s_i\in \mathfrak S_\lambda$.

Recall that  $\sigma$ is the anti-involution on $\mathscr B_{r, s}$   given in Lemma~\ref{inv}. If
$\s, \t\in \Std(\lambda)$ with $\s=(\s_1, \s_2)$ and $\t=(\t_1,
\t_2)$, we define
\begin{equation}\label{nst}
\n_{\s\t}=\sigma(g_{d(\s)} ) \n_\lambda  g_{d(\t)},
\end{equation}
where $\n_\lambda=\n_{\lambda^{(1)}}
\n_{\lambda^{(2)}}$, $g_{d(\s)}=g_{d(\s_1)} g^*_{d(s_2)}$, $g_{d(\t)}=g_{d(\t_1)} g^*_{d(t_2)}$,
$d(\s)=d(\s_1) d(\s_{2})$, and $d(\t)=d(\t_1) d(\t_{2})$.
We remark that we use $s_1, \cdots, s_{r-1}$ and $s_1^*, \cdots, s_{s-1}^*$ to denote generators of $\mathfrak S_r$ and $\mathfrak S_s$, respectively.

Fix $r, s\in \mathbb Z^{>0}$ and $f\in \mathbb N$  with $f\le \min \{r, s\}$. Let
$$ \mathscr{D}_{r,s}^f=\{ s_{f,i_f} s^*_{f, j_f} \cdots
s_{1,i_1}s^*_{1,{j_1}}| ~ k \le {j_k},
 1 \le i_1< i_2 < \cdots <i_f\le r \}.$$
For each $(f, \lambda)\in \Lambda_{r, s}$, we define  $I(f,
\lambda)=\Std(\lambda)\times \mathscr D^f_{r, s}$.

In \cite{Rsong}, we defined $e_{ i, j} = g_{ 1, i }^{-1}
 g_{j, 1 }^*  e_1{g_{ 1, i}} ({g^*_{ j,1}})^{-1}$, for all $1\le i\le r$ and $1\le j\le s$. If $i=j$, we denote $e_{i,j}$ by $e_i$. For any positive
 integer $f\le \min\{r, s\}$, let $e^f=e_1e_2\cdots e_f$. If $f=0$, we denote $e^f$ by $1$.

For any $(\s
,e), (\t, d)\in I(f, \lambda)$, we define
\begin{equation}\label{cellbasis} C_{(\s, e)(\t, d)}=\sigma(g_e)
e^f\n_{\s\t}g_d.\end{equation}

The following result, which has been proved in \cite[Theorem~3.7]{Rsong} can also be obtained from \cite[Theorem~6.13]{ Enyang2}\footnote{Enyang~\cite{Enyang2} has proved that any cellular basis
of Hecke algebras can be lifted to get a cellular basis of $\mathscr B_{r, s}$.}

\begin{Theorem} \label{celb}Let $\mathscr
{B}_{r,s}$ be the quantized walled Brauer algebra over $R$. Then
$\mathcal C$ is a cellular $R$-basis of $\mathscr {B}_{r,s}$ over
the poset $\Lambda_{r, s}$ in the sense of \cite{GL}, where $$\mathcal C=\cup_{(f, \lambda)\in
\Lambda_{r,s}} \{C_{(\s, e)(\t, d)}\mid (\s, e), (\t, d) \in I(f,
\lambda)\}.$$  The required anti-involution $\sigma$ is the one
given in Lemma~\ref{inv}.\end{Theorem}

Recall that $\kappa$ is a field which is an $R$-algebra and $\mathscr B_{r, s, \kappa}=\mathscr B_{r, s}\otimes_R \kappa$.
In this paper, we consider right $\mathscr B_{r, s, \kappa}$-modules.
By standard results on the representations of cellular algebras in \cite{GL},   we have the right cell module
$C(f,\lambda)$ for each  $(f, \lambda)\in \Lambda_{r, s}$, which is
spanned by $\{e^f\n_{\t^\lambda \s} g_d+
\mathscr{B}_{r,s, \kappa}^{\rhd(f,\lambda)}|(\s, d) \in \Std(\lambda) \times
\mathscr{D}_{r,s}^f\}$ as $\kappa$-space, where $\mathscr{B}_{r,s, \kappa}^{\rhd(f,\lambda)}$ is a subspace of $\mathscr B_{r, s, \kappa}$ spanned by $\cup_{(\ell, \mu)\in
\Lambda_{r,s}} \{C_{(\s, e)(\t, d)}\mid (\s, e), (\t, d) \in I(\ell,
\mu)\}$ with $(f, \lambda)\lhd (\ell, \mu)$. In fact,  $\mathscr{B}_{r,s, \kappa}^{\rhd(f,\lambda)}$ is a two-sided ideal of $\mathscr B_{r, s, \kappa}$.

For bipartition $\lambda=(\lambda^{(1)}, \lambda^{(2)})$, let $\lambda'=(\mu^{(1)}, \mu^{(2)})$ where
$\mu^{(i)}$ is the conjugate of  $\lambda^{(i)}$ for $i=1, 2$. We call $\lambda'$ the conjugate of  $\lambda$.
We set $\m_\lambda=\m_{\lambda^{(1)}}\m_{\lambda^{(2)}}$. Note that the current $\m_{\lambda^{(i)}}$ is obtained from
usual one by using $g_{f+j}$ (resp. $g_{f+j}^*$) instead of $g_j$ (resp. $g_j^*$) if $i=1$ (resp. $i=2$).
The following result, which will be needed in section~4, has been proved in \cite{Rsong}.

\begin{Prop} \label{classcell} For each $(f, \lambda)\in \Lambda_{r, s}$, let $\tilde C(f, \lambda):=e^f \m_{\lambda'}g_{d(\t_{\lambda'})}
 \n_\lambda \mathscr B_{r,s} \pmod{ \mathscr \B_{r,s}^{f+1}}$, where
$\mathscr{B}_{r,s}^{f+1}$  is  the two-sided ideal of
$\mathscr{B}_{r,s}$ generated by $e^{f+1}$. As right $\mathscr B_{r, s}$-modules,   $C(f, \lambda)\cong \tilde C(f, \lambda)$.
\end{Prop}

It follows from standard results on the representations of cellular algebras in \cite{GL} that there is an invariant form, say $\phi_{f, \lambda}$, on each cell module
$C(f, \lambda)$. Let $D^{f, \lambda}=C(f, \lambda)/\Rad \phi_{f, \lambda}$, where $\Rad \phi_{f, \lambda}$ is the radical of $\phi_{f, \lambda}$.
Then  $D^{f, \lambda}$ is either zero or absolutely irreducible, and all non-zero $D^{f, \lambda}$'s form a complete set of  all non-isomorphic irreducible $\mathscr B_{r, s, \kappa}$-modules.

Recall that a partition $\lambda$ is called $e$-restricted  if $\lambda_i-\lambda_{i+1}<e$ for all $i\ge 1$.  If $\lambda=(\lambda^{(1)}, \lambda^{(2)})$, then $\lambda$ is $e$-restricted if and only if both $\lambda^{(1)}$ and $\lambda^{(2)}$ are
$e$-restricted.   If $\lambda'$ is $e$-restricted, then $\lambda$ is called $e$-regular.
In \cite{Rsong}, we have proved that $D^{f, \lambda}\neq 0$ if and only if $\lambda$ is $e$-restricted provided that one of the conditions holds: (a) $\delta\neq 0$, (b)
 $\delta =0$ and $r\neq s$, (c) $\delta=0$, $r=s$ and $f\neq r$. This enables us to prove the following result in \cite{Rsong}.
\begin{Theorem}\label{main1}\cite[Theorem~5.3]{Rsong}
Let $\mathscr{B}_{r,s, \kappa}$ be the quantized walled Brauer algebra over
the field $\kappa$.
\begin{enumerate}\item If either $\delta\neq 0$ or $\delta=0$ and $r\neq s$,
then the non-isomorphic irreducible $\mathscr{B}_{r,s, \kappa}$--modules are
indexed by $\{(f, \lambda)\mid  0\le f\le \min\{r, s\}, \lambda \text{ being $e$-restricted}\}$.
\item If  $\delta=0$ and $r=s$,  then the non-isomorphic irreducible
$\mathscr{B}_{r,s, \kappa}$--modules are indexed by  $\{(f, \lambda)\mid  0\le f<r, \lambda \text{ being $e$-restricted}\}$.
\end{enumerate}
\end{Theorem}

We denote by  $[C(f, \lambda): D^{\ell, \mu}]$ the multiplicity of  $D^{\ell, \mu}$ in  $C(f, \lambda)$.
 Such a non-negative integer will be called a decomposition number of  $\mathscr B_{r, s, \kappa}$.

In the remaining part of this paper, we establish the explicit relationship between decomposition  numbers of $\mathscr B_{r, s}$  and those for
 Hecke algebras and $q$-Schur algebras. Using  Ariki, Varagnolo-Vasserot's
results on decomposition numbers of Hecke algebras and $q$-Schur algebras in \cite{Ar, VV} yields the formulae on the decomposition  numbers of $\mathscr B_{r, s}$, as required if the ground field is $\mathbb C$.

\section{Decomposition numbers of $\mathscr B_{r,s, \kappa}$ with $\rho^2\not\in q^{2\mathbb  Z}$}

In this section, we consider $\mathscr B_{r, s, \kappa}$ over $\kappa$ such that   $\rho^2\neq q^{2a}$ for all $a\in \mathbb Z$
with $|a|\le r+s-2$. So, $\rho^2\neq 1$ and $\delta\neq 0$. By Theorem~\ref{semimain},  $\mathscr B_{r, s, \kappa}$ is  semisimple if $e>\max\{r, s\}$. So,
we consider $\mathscr B_{r, s, \kappa}$ under the assumption  $e\le \max\{r, s\}$.  In this case,
 we will prove  that  decomposition numbers of  $\mathscr
B_{r,s, \kappa}$  are determined by those for Hecke algebras
associated with certain symmetric groups.
 We remark that  blocks of $\mathscr B_{r, s, \kappa}$
will also be classified.

Let $\mathscr B_{r, s, \kappa}$-mod be the category of right $\mathscr B_{r, s, \kappa}$-modules.
In \cite[\S4]{Rsong}, we define the exact functor $\F_{r, s}: \mathscr B_{r, s, \kappa}{\text{-mod}} \rightarrow \mathscr B_{r-1, s-1, \kappa}{\text{-mod}}$ and right exact
functor $\G_{r, s}:  \mathscr B_{r, s, \kappa}{\text{-mod}} \rightarrow \mathscr B_{r+1, s+1, \kappa}{\text{-mod}}$\footnote{In \cite{Rsong}, we considered two functors $\F_{r,s}$ and $\G_{r, s}$
for left modules. However, one can prove  similar results for right modules.}. We call $\F_{r, s}$ the Schur functor. By abuse of notations, we use $\F$ and $\G$ instead of $\F_{r, s}$ and $\G_{r, s}$,  respectively.
We remark that we consider right cell modules of $\mathscr B_{r,s, \kappa}$ in  this section.

\begin{Lemma}\label{compo} Let $(\ell, \mu), (f, \lambda)\in \Lambda_{r,
s}$ with $ \mu$ being $e$-restricted.  Then  $[C(f, \lambda):
D^{\ell, \mu}]\neq 0$ only if
 $\ell=f$.
\end{Lemma}
\begin{proof} We prove our result by induction on $r+s$. Since we
are assuming $r, s\in \mathbb Z^{>0}$, we have $r+s\ge 2$. It is not
difficult to check the result for $r+s=2$. In this case, $r=s=1$, $f=1$ and $\lambda=(\emptyset, \emptyset)$ if $f\neq 0$.

 In general, we can assume  $r\ge 2$. Suppose  $\ell>0$. We
apply the exact functor $\F$  to   both $D^{\ell, \mu}$ and $C(f,
\lambda)$.  By \cite[Lemma~4.3]{Rsong}, we have $\F(C(f, \lambda))\cong C(f-1, \lambda)$ for left cell modules. In fact, this holds for
right cell modules.
Since we are assuming  $[C(f, \lambda): D^{\ell,
\mu}]\neq 0$, we have
 $f\ge \ell\ge 1$.
By
\cite[6.2g]{Gr}, $\F(D^{\ell, \mu})$ is either zero or a simple
$\B_{r-1, s-1, \kappa}$-module and each simple $\B_{r-1, s-1, \kappa}$-module is of
 form    $\F(D^{\ell, \mu})$ for some simple $\mathscr B_{r,
s, \kappa}$-module  $D^{\ell, \mu}$. Mimicking arguments in the proof of \cite[Lemma~2.9]{RS2}, we see that
there is a non-trivial homomorphism from $C(\ell-1, \mu)$ to $\F(D^{\ell, \mu})$, forcing   $\F(D^{\ell, \mu})\neq 0$.
 By Theorem~\ref{main1}, $D^{\ell-1, \mu}\neq 0$ and $\F(D^{\ell, \mu})=D^{\ell-1, \mu}$,
if $\ell\ge 1$. By the exactness of $\F$, we
have \begin{equation}\label{decompeq} [C(f-1, \lambda) : D^{\ell-1,
\mu}]=[C(f, \lambda) : D^{\ell, \mu}]\neq 0.\end{equation} Using
induction   assumption yields $f=\ell$.

Now, we assume $\ell=0$. If $f=0$, there is nothing to be proved. So, we assume $f\ge 1$.
Let  $\Res^L M$  be the restriction of $\mathscr B_{r, s, \kappa}$-module  $M $ to $\mathscr B_{r-1, s, \kappa}$.
Note that $D^{0, \mu}$ can be considered as $\mathscr H_r\otimes \mathscr H_s$-module. So, $\Res^L D^{0, \mu}$
can be considered as $\mathscr H_{r-1}\otimes \mathscr H_s$-module.
By  modular branching rule for $\mathscr H_r$ in \cite{Brun}, we  find a removable node
$p$ of $\mu^{(1)}$ such that $D^{0, (\mu^{(1)}\setminus p, \mu^{(2)})}$
is in the socle of $\Res^L D^{0, \mu}$.
It is a composition factor
of $\Res^L C(f, \lambda)$. By \cite[Theorem~4.15]{Rsong} and induction assumption, $f=1$ and
 $$[C(0, (\lambda^{(1)}, \lambda^{(2)}\cup p_2)):  D^{0,
(\mu^{(1)}\setminus p, \mu^{(2)})}]\neq 0   $$ for some  $p_2\in \mathscr A(\lambda^{(2)})$.
 Applying \cite[Lemma~6.3]{Rsong}  to
$C(0, \mu)$, $C(0, (\mu^{(1)}\setminus p, \mu^{(2)}))$,
$C(1, \lambda)$ and $C(0, (\lambda^{(1)},
\lambda^{(2)}\cup p_2))$ yields $\rho^2=q^{2k}$ with
$|k|=|\res(p)+\res(p_2)|\le r+s-2$, where $\res(p)=j-i$ if $p$ is in $i$-th row and $j$-th column. This is  a contradiction.
\end{proof}


Graham and Lehrer\cite{GL}  defined  a cell block
of a cellular algebra, which  is a equivalent class generated by  the notion of cell linked. In our case, $(f, \lambda)$
and $(\ell, \mu)$ are said to be cell linked if either $D^{f, \lambda}$ is a composition factor of $C(\ell, \mu)$
or $D^{\ell, \mu}$ is a composition factor of $C(f, \lambda)$.  By \cite[3.9.8]{GL}, a  block of irreducible modules for $\mathscr B_{r, s, \kappa}$
 is the intersection of $\bar \Lambda_{r, s}$ with a cell block,  where $\bar \Lambda_{r, s}$ consists of all $(f, \lambda)\in \Lambda_{r, s}$ with $D^{f, \lambda}
\neq 0$. See Theorem~\ref{main1} for the explicit description on $\bar \Lambda_{r, s}$.

\begin{Theorem}\label{blocks1} Suppose $(f, \lambda), (\ell, \mu)\in \Lambda_{r, s}$.
 \begin{enumerate}\item  $C(f,
\lambda)$ and $C(\ell, \mu)$ are in the same block if and only if
$f=\ell$ and $C(0, \lambda)$ and $C(0,  \mu)$ are in the same block.
\item
$[C(f, \lambda): D^{\ell, \mu}]=\delta_{f, \ell} [C(0,
\lambda):D^{0, \mu}]$ for any $\mu$ being $e$-restricted.
\end{enumerate}\end{Theorem}

\begin{proof}  We remark that (b) follows from Lemma~\ref{compo} and (\ref{decompeq})
immediately.  We prove (a) as follows.

Without loss of any generality, we can assume that $C(f, \lambda)$ has the simple head $D^{f, \lambda}$ which is a composition factor of $C(\ell, \mu)$.
By Lemma~\ref{decompeq}, $f=\ell$.
Applying exact functor $\F$ to both   $D^{f,
\lambda}$  and $C(\ell, \mu)$ repeatedly, we have that
  $D^{0, \lambda}$ is a composition factor of  $C(0,
\mu)$.

Conversely, let  $D^{0, \lambda}$ be a composition factor of  $C(0, \mu)$. Then, there are two submodules $M_1, M_2$ of $C(0, \mu)$ such that
$D^{0, \lambda}\cong M_1/M_2$.
By the right exactness of $\G$,  there is an epimorphism from $\G(M_1)$ to $G(D^{0, \lambda})$. Similarly, we have an epimorphism from $\G(C(0, \lambda))$ to   $G(D^{0, \lambda})$. By  \cite[Lemma~4.3a]{Rsong},
$D^{0, \lambda}=\F\G(D^{0, \lambda})$, forcing   $G(D^{0, \lambda})\neq 0$. Since $C(0, \lambda)$ has the simple head $D^{0, \lambda}$, $\lambda$ is $e$-restricted. By Theorem~\ref{main1},
 $C(1, \lambda)$ has the simple $D^{1, \lambda}$, forcing an epimorphism from  $G(D^{0, \lambda})$ to  $D^{1, \lambda}$. So, $D^{1, \lambda}$  is a composition factor of
 $\G(M_1)\subset C(1, \mu)$.
 Using the previous arguments repeatedly, we have that $D^{f, \lambda}$  is a composition factor of $C(f, \mu)$. So, $C(f, \lambda)$ and $C(f, \mu)$ are in the same block.

\end{proof}

By Theorem~\ref{blocks1} and   explicit description on  blocks
of Hecke algebras associated with symmetric groups in, e.g.
\cite{Ma}, we know  explicitly the description of blocks of
non-semisimple $\mathscr B_{r,s, \kappa}$ under the assumption $\rho^2\not\in  q^{2\mathbb Z}$. Further, since $C(0, \lambda)$ can be considered as the cell module
of $\mathscr H_{r-f}\otimes \mathscr H_{s-f}$, $[C(f, \lambda): D^{\ell, \mu}]$ can be computed by Ariki's result~\cite{Ar} on the decomposition numbers of Hecke algebra associated to symmetric groups
if the ground field is $\mathbb C$.  More explicitly, such decomposition numbers are computed via  inverse Kazhdan-Lusztig polynomials associated to certain affine Weyl groups of type $A$.

\section{Singular vectors of the  mixed tensor product}\label{rational}
Throughout, let $\mathbb Q(q)$ be  the quotient field of $\mathcal Z$, where  $\mathcal Z=\mathbb Z[q,q^{-1}]$ is  the ring of Laurent polynomials in
indeterminate $q$.

Let $P^\vee $ be the free $\mathbb{Z}$--module with basis
$h_1,\cdots, h_n$, and let  $P^{\vee^*}$ be  its  dual. Then  $P^{\vee^*}$ has a dual basis $\varepsilon_1,\cdots,\varepsilon_n$
 such that $\varepsilon_i(h_j) =\delta_{i,j}, \text{
$1\le i, j\le n$ }$.
The quantum general linear group $\U_q(\mathfrak{gl}_n)$
 is an
associative $\mathbb{Q}(q)$--algebra   generated by $E_i, F_i$,
$1\le i\le n-1$ and $q^h$, $h\in P^\vee$ subject to the defining
relations
\begin{enumerate} \item $ q^0=1$, $q^h q^{h'}=q^{h+h'}$, for any $h,
h'\in P^{\vee}$,
\item $q^h E_i q^{-h}=q^{\alpha_i(h)} E_i$, where  $\alpha_i=\varepsilon_i-\varepsilon_{i+1}$,
\item $q^h F_i q^{-h}=q^{-\alpha_i(h)}F_i$,
\item $E_iF_j-F_jE_i= \delta_{i,j} \frac{K_i-K_i^{-1}}{q-q^{-1}}$ and $K_i=
q^{h_i-h_{i+1}}$,
\item $E_iE_j=E_jE_i$, for $|i-j|>1$,
\item $F_iF_j=F_jF_i$, for $|i-j|>1$.
\item $E_i^2E_j-(q+q^{-1})E_iE_jE_i+E_jE_i^2=0$
for $|i-j|=1$, \item $F_i^2F_j-(q+q^{-1})F_iF_jF_i+F_jF_i^2=0$  for
$|i-j|=1$, \end{enumerate}

It is well known that
$\U_q(\mathfrak{gl}_n)$ is a Hopf algebra   such that the
 comultiplication $\Delta$, counit $\epsilon$ and antipode $S$ satisfy the following conditions:
\begin{enumerate}\item $\Delta(E_i)=E_i\otimes K_i^{-1} +1\otimes E_i$,
\item $\Delta(F_i)=F_i\otimes 1+K_i\otimes F_i$,
\item $\Delta (q^h)=q^h\otimes q^h$,
\item $S(F_i)=-K_i^{-1} F_i$, $S(E_i)=-E_i K_i$ and $S(q^h)=q^{-h}$,
\item $ \varepsilon(E_i)=\varepsilon(F_i)=0$ and $\varepsilon (q^h)=1$.
\end{enumerate}

If we use $q^{-1}$ instead of $q$, then the previous $\U_q(\mathfrak {gl}_n)$ is the quantum general
linear group in \cite{Jan} and the current $E_i$ and $F_j$ correspond to $F_i$ and $E_j$ in \cite{Jan}.

It is known that $\U_q(\mathfrak{gl}_n)$ has a  $\mathcal Z$-Hopf-subalgebra $\mathbf U_{\mathcal Z}$, which is generated by
 $q^h$, and divided powers $E_i^{(\ell)}=\frac {E_i^\ell}
{[\ell]!}$ and $F_i^{(\ell)}=\frac {F_i^\ell} {[\ell]!}$, for all $h\in P^{\vee}$
and all $\ell\in \mathbb Z^{>0}$, where $[\ell]!=[\ell][\ell-1]\cdots
[1]$, $[\ell]=\frac{q^\ell -q^{-\ell}}{q-q^{-1}}$.

For
each left $\U_{\mathcal Z}$-module $M$, and  $\lambda\in \mathbb
Z^n$, define
$$M_\lambda=\{m\in M\mid q^{h_i}\cdot m=q^{\lambda_i} m, 1\le i\le n\}.$$
Then  $\lambda$ is called   a weight of $M$ if $M_\lambda\neq 0$. In this case, $M_\lambda$ is called the weight space  of $q^h$ acting on $M$.
  Further, each weight space  of $M$ is of form $M_\lambda$.
If $\lambda_1\ge \lambda_2\ge \cdots\ge \lambda_n$, then
$\lambda$ is called a \textsf{dominant  weight}. Let $\mathfrak
{X}^+(n)$ be the set of all $\lambda\in \mathbb Z^n$ with $\lambda_i\ge \lambda_{i+1}$ for all $i$, $1\le i\le n-1$.

\begin{Lemma}\label{natural1}\cite{DDS1}
Let $V$ be the  free $\mathcal Z$-module $V$ with basis $\{v_1, v_2, \cdots,
v_n\}$.  Let $V^*=\text{Hom}_{\mathcal  Z} (V, \mathcal
Z)$ be the dual of $V$ with dual basis  $\{v_1^*, v_2^*, \cdots,
v_n^*\}$. Then both  $V$ and $V^*$ are  left $\U_{\mathcal Z}$-module  such that
\begin{enumerate}\item  $ q^h v_j  =q^{\varepsilon_j(h)}v_j$,  $E_iv_j=\delta_{j, i+1}v_i$, $F_iv_j=\delta_{i, j} v_{i+1}$,
\item $  q^h v_j^*  =q^{-\varepsilon_j(h)}v_j^*$,
$E_iv_j^*=-\delta_{i, j} q^{-1}v_{i+1}^*$,
$F_iv_j^*=-\delta_{i+1, j} q v_{i}^*$,\end{enumerate}
  if all of them make sense. \end{Lemma}
\begin{proof} (a) has been given in \cite{DDS1} and (b) can be verified easily by using (a) and antipode $S$ for $\U_q(\mathfrak{gl}_n)$.\end{proof}

In the remaining part of this paper, we denote $V^*$ by $W$. Fix two positive integers $r$ and $s$. Then
the mixed tensor space $V^{r, s}: =V^{\otimes r}\otimes W^{\otimes s}$, which was studied in \cite{DDS1, DDSII}, is a left $\U_{\mathcal Z}$-module.
Given positive integers $n, r, s$, let
\begin{equation}\begin{aligned} & I(n, r) =\{\mathbf i\mid \mathbf i=(i_r, i_{r-1},
\cdots, i_1),   1\le i_j\le n,  1\le
j\le r\},\\
& I^*(n, s) =\{\mathbf i\mid \mathbf i=(i_1, i_{2},
\cdots, i_s),   1\le i_j\le n,  1\le
j\le s\}.\\
\end{aligned}
\end{equation}
Then the symmetric group $\mathfrak S_r\times \mathfrak S_s$  acts on the right of $I(n, r)\times I^*(n,s)$ by place permutation in the sense $(\mathbf i, \mathbf j)ww^*=(\mathbf iw, \mathbf jw^*)$ for
any $(\mathbf i, \mathbf j)\in I(n, r)\times I^*(n,s)$ and $w\in  \mathfrak S_r$, and $ w^*\in \mathfrak S_s $.

For each $(\mathbf i, \mathbf j)\in I(n, r)\times I^*(n,
s)$, define  $$\lambda_k=\# \{\ell \mid i_\ell=k\}-\# \{\ell \mid j_\ell=k\},$$ for  $1\le k\le n$ and
write $\text{wt}(\mathbf i, \mathbf j)=(\lambda_1, \cdots, \lambda_n)$.  We
 call $(\lambda_1, \cdots, \lambda_n)\in \mathbb Z^n$, the weight of $(\mathbf i, \mathbf j)$. It is easy to see that $(\mathbf i, \mathbf j)$ and $(\mathbf k, \mathbf l)$ have the same weight if they  are in the same $\mathfrak S_r\times \mathfrak S_s$-orbit. However, the converse is not true.

For each  $(\mathbf i, \mathbf j)\in I(n, r)\times I^*(n,
s)$, define   $v_{\mathbf i|\mathbf j}=v_{\mathbf i}\otimes
v^*_{\mathbf j}$, where
\begin{equation}\label{ij}  v_{\mathbf i}=v_{i_r}\otimes
v_{i_{r-1}}\otimes \cdots\otimes  v_{i_1}, \quad \text{and}\quad
v^*_{\mathbf j}=v^*_{j_1}\otimes v^*_{j_2}\otimes \cdots\otimes
v^*_{j_s}.
\end{equation}
Then $\{v_{\mathbf i|\mathbf j}\mid (\mathbf i, \mathbf j) \in I(n,
r)\times I^*(n, s)\}$ is a $\mathcal Z$-basis of $V^{r, s} $.
Obviously, the weight of $v_{\mathbf i|\mathbf j}\in V^{r, s}$ is the same as the weight of  $(\mathbf i, \mathbf j)\in I(n, r)\times I^*(n, s)$.

 \begin{Lemma}~\label{bij} Let $\Lambda(r,s)=\{\lambda\in\mathbb Z^n \mid \sum_{\lambda_i>0} \lambda_i=r-f,
\sum_{\lambda_j<0} \lambda_i=f-s, 0\le f\le \text{min} \{r, s\}\}$. Then  $\Lambda(r, s)$ is the
set of weights of  $V^{r, s}$.
 \end{Lemma}
\begin{proof} Easy exercise.\end{proof}

\begin{Lemma}\label{phila} Given $r, s, n\in \mathbb Z^{>0}$ with $n\ge r+s$, let  $\Lambda^+(r, s)=\Lambda(r, s)\cap \mathfrak X^+(n)$.
\begin{enumerate}\item There  is a bijection  $\phi: \Lambda_{r, s}\rightarrow \Lambda^+(r, s)$;
 \item If $\Lambda=\{\lambda+s\omega\mid
\lambda\in \Lambda^+(r, s)\}$ with  $\omega=(1, 1, \cdots, 1)\in
\mathbb Z^n$, then $\Lambda$  is an
 ideal of $\Lambda^+(n,
r+(n-1)s)$ in the sense that $\lambda\in \Lambda$ if
$\lambda\trianglelefteq \mu$ for some $\mu\in \Lambda$.\end{enumerate}
\end{Lemma}

\begin{proof} Suppose  $(f, \lambda)\in \Lambda_{r, s}$ such that
$\lambda=(\lambda^{(1)}, \lambda^{(2)})$,  $l(\lambda^{(1)})=k$ and
$\l(\lambda^{(2)})=\ell$, where $l(\lambda^{{(i)}})$ is the maximal index $j$ such that  $\lambda^{(i)}_j\neq 0$.   Since $n\ge r+s$,  the required bijection $\phi$  sends
$(f, \lambda)$ to $\phi(f, \lambda)$ where
$$\phi(f, \lambda) =(\underset{n} {\underbrace{\lambda^{(1)}_1, \lambda^{(1)}_2, \cdots,
\lambda^{(1)}_{k},0, \cdots, 0, -\lambda^{(2)}_{\ell}, \cdots,
-\lambda^{(2)}_1})}.$$  This proves (a). One can verify (b) by
straightforward computation. \end{proof}

Let $\kappa$ be a field which is a $\mathcal Z$-algebra, let $\U_\kappa=\U_{\mathcal Z}\otimes_{\mathcal Z} \kappa$.
By abuse of notations, we denote $E_i^{(\ell)}$ (resp.,  $F_i^{(\ell)}$ ) by $E_i^{(\ell)}\otimes 1_\kappa $ (resp.,  $F_i^{(\ell)}\otimes 1_\kappa$).

Suppose  $M$ is  a finite dimensional  left $\U_{ \kappa}$-module.  If $0\neq v\in M_\lambda$ for some $\lambda\in \mathfrak {X}^+(n)$,   such that $E_i^{\ell}/[\ell]! v=0$, $\forall i, \ell, 1\le i\le
n-1$ and $\ell>0$, then $v$ is called a \textsf { highest weight } (or \textsf{singular}) vector of $M$ with highest weight $\lambda$.

In the remaining part of this section, we
want to classify singular vectors of  $V^{r, s}_\kappa =V^{\otimes r}\otimes W^{\otimes s}\otimes \kappa\cong V_\kappa^{\otimes r}\otimes W_\kappa^{\otimes s}$ over $\kappa$,
provided that $n\ge r+s$. Since we are going to use Dipper-Doty-Stoll's results in \cite{DDS1, DDSII}, we consider their presentation of  $\mathscr B_{r, s}$ with  $\rho=q^{n}$ over $\mathcal Z$.
As mentioned before, Dipper-Doty-Stoll's presentation for $\mathscr B_{r, s}$ can be obtained from that in section~2 by using  $q^{-1}$ and $\rho^{-1}$ instead of
$q$, $\rho$, respectively.  In this case, we still have $\rho^{-1}=(q^{-1})^n$.

\begin{Prop}\cite{DDS1}\label{bmodule}  Let $\mathscr B_{r, s}$ be the quantized walled Brauer algebra over $\mathcal Z$ with defining parameter $\rho=q^{n}$.
Then  $V^{ r, s}$ is a right
${\mathscr{B}}_{r,s}$-module over $\mathcal Z$ such that, for any  $(\mathbf i, \mathbf j) \in I(n,
r)\times I^*(n, s)$,

\begin{enumerate} \item
$v_{\mathbf i|\mathbf j}
e_1=\delta_{i_1,j_1}q^{-n-1+2i_1}\sum_{s=1}^n  v_{\hat{\mathbf
i}}\otimes v_{s}\otimes v_{s}^*\otimes v_{\hat {\mathbf j}}^*$,
\item $ v_{\mathbf i|\mathbf j}g_{k}=q^{-1} v_{\mathbf i|\mathbf j}$, ( resp.,   $v_{\mathbf i|\mathbf j}g_k^*=q^{-1} v_{\mathbf i|\mathbf j}$), if $i_{k}=i_{k+1}$, (resp., $j_k=j_{k+1}$),
\item $ v_{\mathbf i|\mathbf j}g_{k}=v_{\mathbf i s_{k}|\mathbf j}$, (resp., $v_{\mathbf i|\mathbf j}g_k^*=v_{\mathbf i|\mathbf j s^*_{k}}$ ) if  $i_{k}<i_{k+1}$, (resp., $j_k>j_{k+1}$),
\item $ v_{\mathbf i|\mathbf j}g_{k}=v_{\mathbf i s_{k}|\mathbf j}+(q^{-1}-q)v_{\mathbf i|\mathbf j}$, (resp.  $v_{\mathbf i|\mathbf j}g_k^*=v_{\mathbf i|\mathbf j s^*_{k}}+(q^{-1}-q)v_{\mathbf i|\mathbf j}$),
if  $i_{k}>i_{k+1}$, (resp., $j_k<j_{k+1}$),
\end{enumerate}
where $\hat{\mathbf i}$ (resp., $\hat{\mathbf j})$ is obtained
from $\mathbf i$ (resp., $\mathbf j$) by dropping $i_1$ (resp.,
$j_1$).\end{Prop}

\begin{rem} In this paper,  $v_{\mathbf i}=v_{i_r}\otimes\cdots \otimes  v_{i_1}$ for $\mathbf i\in I(n, r)$, whereas $v_{\mathbf i}=v_{i_1}\otimes\cdots \otimes v_{i_r}$ in  \cite{DDS1}.
  In \cite[p6]{DDS1}, Dipper, Doty and Stoll  defined operators $E, S_i, \hat S_j$ acting on $V^{r, s}$. It is pointed in \cite[Corollary~1.9]{DDSII} that $e_1$, $g_i$ and $g_j^*$ act on
$V^{r, s}$ via $E, S_{r-i}$ and $\hat S_{r+j}$, respectively. So, the condition for $g_k$ acting on $V^{r, s}$ is the same as that for $S_k$ in \cite{DDS1}.\end{rem}


\begin{Theorem} \cite[Theorem~1.4]{DDS1},\cite[Theorem~6.1, Corollary~6.2]{ DDSII}\label{ration} Suppose $r, s\in \mathbb Z^{> 0}$. Then $V^{r, s}$ is a ($\U_{\mathcal Z},
\mathscr B_{r,s})$-bimodule.  Moreover, \begin{enumerate}
\item there is an algebra  epimorphism $\varphi: \U_{\mathcal Z}\twoheadrightarrow \text{End}_{\mathscr B_{r, s}} ( V^{r,s})$;
\item there is an
algebra epimorphism $\psi: \mathscr B_{r, s}\twoheadrightarrow
\text{End}_{\U_\mathcal Z} ( V^{r, s})$. Further, $\psi$ is an
isomorphism if and only if  $n\ge r+s$.
\end{enumerate}
\end{Theorem}

In particular, Theorem~\ref{ration} holds over an arbitrary field $\kappa$.
The endomorphism algebra $\text{End}_{\mathscr B_{r, s}} ( V^{r,s})$ , which will be  denoted by $S(n; r,s)$, is called the
 rational $q$-Schur algebra in \cite{DDS1}.

 In the remaining part of this section,  unless otherwise stated, we assume $n\ge r+s$.
 We want to classify singular vectors of $V_\kappa^{r, s}$ with $n\ge r+s$ so as to establish explicit relationship between
 Weyl modules, indecomposable tilting modules of rational $q$-Schur algebras and cell modules, principal indecomposable modules of
 $\mathscr B_{r, s}$. This will give the required result on the  decomposition numbers of $\mathscr B_{r, s, \kappa}$.
  If we allow $s=0$, then
 $S(n; r,s)$ is known as $q$-Schur algebra $S(n, r)$ in \cite{DJ}.

In \cite{Dk1}, Donkin has proved that rational Schur algebras in \cite{DD} are
generalized Schur algebras. Note that generalized Schur algebras are always
quasi-hereditary in the sense of \cite{CPS}. So, rational Schur algebras are quasi-hereditary.
The same is true for their quantizations. It is natural to modify his arguments to prove that rational $q$-Schur algebras are quasi-hereditary over $\kappa$.
We need this fact when we classify singular vectors of $V^{r, s}$. Motivated by  Dipper and  Doty's work on quasi-heredity of   rational Schur algebras in \cite{DD},
 we use arguments on cellular algebras to prove this fact.

 For the simplification of notation,
we  use $\tilde \lambda$
instead of $\phi(f, \lambda)$ in the remaining part of this paper, where $(f, \lambda)\in\Lambda_{r, s}$ and   $\phi(f, \lambda)$ is given in Lemma~\ref{phila}.
 Dipper et.al~\cite{DDS1}
defined column semistandard rational $\lambda$-tableaux $(\s_1, \s_2)$ in \cite[6.1]{DDS1}\footnote{In fact, Dipper et.al  called $(\s_1, \s_2)$ standard rational tableaux, which are different from usual
 standard tableaux in section~2.}.  They proved
that each  $(\s_1, \s_2)$ corresponds to a unique column semistandard
$\gamma$-tableau $\u$ and vice versa, where $\gamma'=s\omega+\tilde
{\lambda'}$. The transpose of  this column   semistandard $\gamma$-tableau is
in fact  the usual semistandard $\gamma'$-tableau.  In the following,
we denote $\gamma$ by $\gamma_\lambda$ to emphasis the  $(f,
\lambda)$ in $\Lambda_{r, s}$.


Let $(\s_1,\s_2), (\t_1,\t_2)$ be two  column semistandard rational  $\lambda$-tableaux.
Dipper etc~\cite{DDS1} introduced rational bideterminants
$((\s_1,\s_2)\mid (\t_1,\t_2))\in A_q(n; r, s)$ where $ A_q(n; r,
s)$ is the linear dual of $S(n; r, s)$. It is proved in
\cite[Theorem~6.9]{DDS1} that the set of all bideterminants of column semistandard rational $\gamma_\lambda$-tableaux
with $(f, \lambda)\in \Lambda_{r, s} $ forms a $\mathcal Z$-basis of $A_q(n; r,
s)$.

Let $ A_q(n,
r+(n-1)s)$ be  the linear dual of  $S(n,
r+(n-1)s)$ in \cite{DJ}.  For column semistandard
$\lambda'$-tableau $\u, \v$ with  $\lambda\in \Lambda^+(n, r+(n-1)s)$, let $(\u\mid \v)\in A_q(n, r+(n-1)s)$ be the
corresponding bideterminant in \cite{DDS1}. In fact, it is the same as the bideterminant in \cite{HZ}, which is defined via  the semistandard $\lambda$-tableaux
$\s, \t$,  the transposes of $\u, \v$, respectively.

It has been proved in~\cite{HZ} that the
set of  all bideterminants $(\u\mid \v)\in A_q(n, r+(n-1)s)$ of column semistandard  $\lambda$-tableaux
$(\u\mid \v)$ with $\lambda'\in \lambda(n, r+(n-1)s)$  forms a
$\mathcal Z$-basis of $A_q(n, r+(n-1)s)$.


Dipper et.al~\cite[6.5]{DDS1} have proved that   $ A_q(n; r, s)$
can be embedded into $A_q(n, r+(n-1)s)$ via the linear map such that
\begin{equation}\label{bidet}
\iota  ((\s_1,\s_2)\mid (\t_1,\t_2))=q^c (\u\mid \v)\end{equation}
for some integer $c$. Here $\u, \v$ are column semistandard
$\gamma_\lambda$-tableaux with $(f, \lambda)\in \Lambda_{r, s}$, which correspond to column semistandard rational $\lambda$-tableaux $(\s_1, \s_2)$ and $(\t_1,
\t_2)$, respectively. In this paper, we do not need the details about this.
\medskip

\begin{Prop}~\cite[Corollary~6.1]{DDS1}\label{schurrational} Let $\pi$ be the linear dual of $\iota$. Then
 $\pi: S(n, r+(n-1)s)\twoheadrightarrow S(n; r,
s)$ is an algebra epimorphism over $\mathcal Z$.  \end{Prop}

Let $\kappa$ be a field which is a $\mathcal Z$-algebra. It is
proved in \cite{DDS1} that  the rational $q$-Schur algebra  over $\kappa$ is isomorphic to $S(n; r,
s)\otimes_{\mathcal Z} \kappa$. For this reason, we identify
$S_\kappa(n; r, s)$ with  $S(n; r, s)\otimes_{\mathcal Z} \kappa$.
The following result follows from  certain results in
\cite{ DDS1}. The classical case has been given in \cite{DD} and \cite{Dk1}. As mentioned before,
it can also follow from  arguments similar to those  in \cite{Dk1}.

\begin{Theorem}\label{qha}  Suppose $n, r, s\in \mathbb Z^{>0}$ with $n\ge r+s$. Then
$S_\kappa(n; r, s)$ is quasi-hereditary over $\kappa$ in the sense
of \cite{CPS}.
\end{Theorem}
\begin{proof}  Suppose $\alpha\in
\Lambda^+(n, r+(n-1)s)$. For usual semistandard $\alpha$-tableaux $\u, \v$,
let $Y_{\u,\v}^\alpha\in S_\kappa(n, r+(n-1)s)$ be the codeterminant in \cite[p~48]{Cli}.

We claim that $\pi(Y_{\u,\v}^\alpha)=0$  if $\alpha\not\in \Lambda $, where $\pi$ is given in Proposition~\ref{schurrational} and $\Lambda$ is the ideal of
$\Lambda^+(n, r+(n-1)s)$ defined in Lemma~\ref{phila}.

In fact, if  $\pi(Y_{\u,\v}^\alpha)\neq 0$, we can find  a $(f, \lambda)\in \Lambda_{r, s}$ such that
$\pi(Y_{\u,\v}^\alpha) ((\s_1, \s_2)|(\t_1, \t_2))\neq 0$ for some rational bideterminant  $(\s_1, \s_2)|(\t_1, \t_2)$ associated
to a pair of column standard rational $\lambda$-tableaux $(\s_1, \s_2)$ and $(\t_1, \t_2)$.
So, $Y_{\u, \v}^\alpha ( \u_1\mid \v_1)\neq 0$ where $$q^c (\u_1\mid \v_1)=\iota (\s_1, \s_2)|(\t_1, \t_2)$$
for some integer $c$ (see  (\ref{bidet})).
Note that $(\u_1, \v_1)$ are a pair of semistandard $\gamma'_\lambda$-tableaux (or column semistandard $\gamma_\lambda$-tableaux if we use
the notion of bideterminants in \cite{DDS1}). By\cite[Theorem~12]{Cli}, we have $\gamma'_\lambda\unrhd \alpha$. Since $\Lambda$ is an ideal of $\Lambda^+(n, r+(n-1)s)$
and  $\gamma'_\lambda\in \Lambda$, we have $\alpha\in \Lambda$,  proving the claim.   Counting the dimension of
$S_\kappa(n; r, s)$, we see that the image of each codeterminant
$Y_{\u,\v}^\alpha$ is nonzero in  $S_\kappa(n; r, s)$ if $\alpha\in
\Lambda$. Further, all non-zero of them form a basis of  $S_\kappa(n; r, s)$ and
the kernel of $\pi$ is the $\kappa$-subspace
generated by all $Y_{\u,\v}^\alpha$ with $\alpha\not\in \Lambda$.

It is proved in \cite[Theorem~5.5.1]{DR1} that  the codeterminant basis of a
q-Schur algebra is a standard  basis in the sense of \cite[Definition~1.2.1]{DR1}.
 One can check that the linear map $\tau$ sending any $Y_{\u,\v}^\alpha$ to $Y_{\v,\u}^\alpha$
is the required anti-involution. So, the codeterminant basis of a
q-Schur algebra   is a cellular basis  in the sense of
\cite{GL}.
Therefore, $S_\kappa(n; r, s)$ is a cellular algebra with the
cellular basis which consists of all the images of codeterminants
$Y_{\u,\v}^{\alpha}$ for $\alpha\in \Lambda$. Further, non-isomorphic
irreducible $S_\kappa(n; r, s)$-modules are indexed by $\Lambda$. By
\cite[3.10]{GL},  $S_\kappa(n; r, s)$ is quasi-hereditary in the
sense of \cite{CPS}.\end{proof}


\begin{Defn}\label{vlambda} For each  $(f, \lambda)\in\Lambda_{r,s}$ with
$\lambda=(\lambda^{(1)}, \lambda^{(2)})$, we write  $\lambda'=(\alpha,
\beta)$.  We  define   $ v_{0, \lambda}=v_{\mathbf i_\lambda} \otimes
v^*_{\mathbf j_\lambda}$,  and $v_{f, \lambda}$  for $f\ge 1$ as
$v_{\mathbf i_\lambda}\otimes v^f \otimes v^*_{\mathbf j_\lambda}\in
V^{r,s}$,  where
\begin{enumerate}\item  $ \mathbf i_\lambda=(\alpha_{r-f}, \cdots, 2, 1,  \alpha_{r-f-1}, \cdots, 2, 1, \cdots, \alpha_{1}, \cdots, 2, 1
)$, \item $ \mathbf j_\lambda=(n, n-1, \cdots, n-\beta_1+1,  \cdots, n, n-1,\cdots, n-\beta_{s-f}+1)$,
\item   $v^f= \sum_{k=1}^n v_k\otimes v^{f-1} \otimes v_k^*$ and
$v^1=\sum_{k=1}^n v_k \otimes v_k^*$.\end{enumerate}
\end{Defn}

\begin{Prop}\label{hwv}  Suppose $V_\kappa^{r, s}$ is defined over $\kappa$.
 For each $\t\in \Std(\lambda')$ with  $(f, \lambda)\in \Lambda_{r, s}$ and  $d\in \mathscr D_{r,s}^f$,
  define $v_{\lambda, \t, d}=v_{f, \lambda}\n_{\lambda'} g_{d(\t)} g_d\in V_\kappa^{r,
s}$. Then $ v_{\lambda, \t, d}\in V_\kappa^{r, s}$ is a singular vector with highest weight $\phi(f, \lambda)$.
 \end{Prop}

\begin{proof}  By Theorem~\ref{ration},
$V^{r, s}$ is a $(\U_{\mathcal Z}, \mathscr
B_{r,s})$-bimodule over $\mathcal Z$, where $V$ is  the free $\mathcal Z$-module with  rank $n$.  If we have $E_i (v_{f,
\lambda} \n_{\lambda'})=0$ over   $\mathbb  Q(q)$, then
$E_i^{\ell}/[\ell]! (v_{f, \lambda} \n_{\lambda'})=0$ over $\mathbb
Q(q)$, forcing $E_i^{\ell}/[\ell]! (v_{f, \lambda} \n_{\lambda'})=0$
over $\mathcal Z$. By base change, $E_i^{\ell}/[\ell]! (v_{f,
\lambda} \n_{\lambda'})=0$ over $\kappa$.

 For any positive integer $j$ with $ j\le f $, let  $$x_1 =
1^{\otimes (r-j)} \otimes E_i\otimes (K_i^{-1})^{\otimes (s+j-1)}, \text{ and   $ x_2 = 1^{\otimes (r+j-1)}\otimes  E_i \otimes
(K_i^{-1})^{\otimes (s-j)}$.}$$  If $v=v_1  \otimes v^{j}  \otimes v_2\in
V^{r,s}$ where $v_1\in V^{\otimes (r-j)}$ and  $v^{j}$ is given in Definition~\ref{vlambda}(c), by  Lemma~\ref{natural1},
$$\begin{aligned} x_1v & =q^{-1}v_1  \otimes v_i \otimes v^{j-1}\otimes v_{i+1}^* \otimes (K_i^{-1})^{\otimes (s-j)}v_2 ,\\
x_2v & =-q^{-1}v_1  \otimes v_i \otimes v^{j-1}\otimes v_{i+1}^* \otimes (K_i^{-1})^{\otimes (s-j)}v_2.\\
\end{aligned} $$ So,  $(x_1+x_2) v=0$. Note that
 $$
\Delta^{r+s-1} (E_i) =\sum_{j=0}^{r+s-1} 1^{\otimes j} \otimes E_i
\otimes (K_i^{-1})^{(r+s-j-1)}. $$
So,   $E_iv_{f, \lambda} $ can be written as a linear
combination of elements $v_{\mathbf i_\lambda^k}\otimes v^f\otimes
v_{\mathbf j_\lambda}^*$ and $v_{\mathbf i_\lambda}\otimes
v^f\otimes v_{\mathbf j_\lambda^k}^*$ where \begin{enumerate} \item  $\mathbf i_\lambda^{k}$
is obtained from $\mathbf
i_\lambda$ by using $i$ instead of
$i+1$ in the sequence $(\alpha_k, \cdots, 2, 1 )$ if  $i\le \alpha_k-1$,
\item  $\mathbf j_\lambda^{k}$  is obtained from   $\mathbf j_\lambda$  by using $i$ instead of  $i-1$
 in the sequence   $(n, n-1, \cdots, n-\beta_k+1)$ if $i\ge n-\beta_k+2$.
 \end{enumerate}
 If $\mathbf i_\lambda^{k}$ is well defined, we write    $w_k=(\sum_{j=1}^{k-1} \alpha_j+i,
\sum_{j=1}^{k-1}\alpha_j+i+1)\in \mathfrak S_{\lambda'}$. So,
$$  v_{\mathbf i_\lambda^k}\otimes v^f\otimes v^*_{\mathbf
j_\lambda}  g_{w_k} \n_{\lambda'}=q^{-1} v_{\mathbf i_\lambda^k}\otimes
v^f\otimes v^*_{\mathbf j_\lambda}\n_{\lambda'}=-q v_{\mathbf
i_\lambda^k}\otimes v^f\otimes v^*_{\mathbf j_\lambda}\n_{\lambda'},
\text{ over $\mathcal Z$.}$$  This implies $v_{\mathbf
i_\lambda^k}\otimes v^f\otimes v^*_{\mathbf
j_\lambda}\n_{\lambda'}=0$ over $\mathbb Q(q)$. Similarly,
$v_{\mathbf i_\lambda}\otimes v^f\otimes v^*_{\mathbf
j_\lambda^k}\n_{\lambda'}=0$. So, $E_i (v_{f, \lambda}
\n_{\lambda'})=0$  over   $\mathbb  Q(q)$. By Theorem~\ref{ration},  $E_i v_{\lambda, \t, d}=(E_i(v_{f, \lambda}
\n_{\lambda'})) g_{d(\t)} g_d=0$.
 Finally, it is easy to see that the weight of  $v_{\lambda, \t, d}$ is $\phi(f, \lambda)$.
\end{proof}

\begin{Lemma}\label{indep} Suppose $(f, \lambda)\in \Lambda_{r, s}$.
 Then $\{v_{\lambda, \t, d}|\t\in \Std(\lambda'), d\in \mathscr{D}_{r,s}^f\}$ is $\kappa$-linearly independent.
\end{Lemma}

\begin{proof}

 Let $\xi_\lambda$ be  obtained from $v_{f, \lambda}$ (see
 Definition~\ref{vlambda}) by using $\mathbf v$ instead of
$v^f$ in Definition~\ref{vlambda}(c), where $\mathbf
v=v_{\alpha_1+1}\otimes\cdots\otimes v_{\alpha_1 +f}\otimes
v_{\alpha_1+f}^*\otimes\cdots\otimes v^*_{\alpha_1+1}$ and   $\lambda'=(\alpha, \beta)$.

If  $\sum_{\t,d}a_{\t,d} v_{\lambda, \t, d}=0$,
$a_{\t,d}\in \kappa$, by Proposition~\ref{bmodule},
$\sum_{\t}a_{ \t,d }\xi_\lambda
\n_{\lambda'}g_{d(\ts)}g_d=0$  for any fixed $d$. Since
$g_d$ is invertible, $\sum_{\t} a_{ \t, d}\xi_\lambda
\n_{\lambda'}g_{d(\t)}=0$.

 It is well known that $V^{\otimes r}\cong \oplus_{\lambda\in \Lambda(n,r)}  \m_\lambda \H_r$ as right $\mathscr H_r$-modules, where $\m_\lambda$ is obtained from
 that in section~2 by using $q^{-1}$ instead of $q$. The corresponding isomorphism sends $v_{\mathbf i_\lambda d}$ to $q^{-l(d)}\m_\lambda g_d$\footnote{If $(g_i-q)(g_i+q^{-1})=0$, then the corresponding
  isomorphism sends $v_{{\mathbf {i}}_{\lambda} d}$ to $q^{l(d)}\m_\lambda g_d$.}  for any
distinguished right coset  representative $d$ of $\mathfrak
S_\lambda/\mathfrak S_r$. In particular, if we consider $V_\kappa^{r-f}\otimes \mathbf v\otimes W_\kappa^{s-f}$
 as  right $\mathscr H_{r-f}\otimes \mathscr H_{s-f}$-module, then
$\xi_\lambda\n_{\lambda'}g_{d(\ts)}$ corresponds to $\m_\lambda g_{w_\lambda} \n_{\lambda'} g_{d(\ts)}$ up to a non-zero scalar in $\kappa$, where $w_\lambda=d(\t_\lambda)$.  Note that $\mathscr H_{r-f}\otimes \mathscr H_{s-f}$
are generated by $g_{f+i}$ and $g^*_{f+j}$ for all positive integers $i, j$ with $f+i\le r-1$ and $f+j\le s-1$.
By \cite[Theorem~5.6]{DJ1}, $a_{\t,
d}=0$, for all possible $\t$ and $d$.
\end{proof}

By Proposition~\ref{hwv}, each
$\U_{\kappa}$-module generated by
 $v_{\lambda, \t, d}$ is a highest weight module with  highest weight
 $\tilde \lambda:=\phi(f, \lambda)$ in Lemma~\ref{phila}.  By the universal property of  Weyl modules in \cite[1.20]{APW},
  $\U_{\kappa} v_{\lambda, \t, d}$ is a quotient of $\Delta(\tilde \lambda)$ where   $\Delta({\tilde \lambda})$ is the Weyl  module  of $\U_{\kappa}$
  with respect to
 the highest weight $\tilde \lambda$.  We will use this fact  in Proposition~\ref{weyl-cell}.

\begin{Prop}\label{weyl-cell} If  $(f, \lambda)\in \Lambda_{r, s}$, then there is an isomorphism
$$\text{Hom}_{\U_\kappa}(\Delta(\tilde \lambda), V_\kappa^{r,s})\cong C(f, \lambda')$$ as right $\mathscr B_{r, s, \kappa}$-modules if $n\ge r+s$.
\end{Prop}
\begin{proof}

For each $(d, \t)\in \mathscr D_{r,s}^f \times \Std(\lambda')$, let
$M_{d, \t}= \U_\kappa v_{\lambda, \t, d}\subset V^{r, s}$. By Proposition~\ref{hwv},
$v_{\lambda, \t, d}\in V^{r, s}$ is a highest weight vector with highest weight
$\tilde \lambda$. So, there is a unique $\U_\kappa$-epimorphism (up to a scalar)  from
$\Delta({\tilde \lambda})$ to $M_{d, \t}$ sending highest weight vector to highest
weight vector. Such a $\U_\kappa$-homomorphism will be denoted by  $f_{\lambda, \t, d}$.
Since  $M_{d, \t}$ is a submodule of $V^{r, s}$,  $f_{\lambda, \t, d}$ results in a homomorphism
in  $\text{Hom}_{\U_\kappa}(\Delta(\tilde \lambda), V_\kappa^{r,s})$. By abuse of notation, we denote this homomorphism by $f_{\lambda, \t, d}$.
By  Lemma~\ref{indep},  $\{f_{\lambda, \t, d}\mid  (d, \t)\in
\mathscr D_{r,s}^f \times \Std(\lambda')\}$ is $\kappa$-linear
independent.

Now, we compute the dimension of   $\text{Hom}_{\U_\kappa}(\Delta(\tilde \lambda), V_\kappa^{r,s})$.
It is known that  $V=\Delta(\epsilon_1)=\nabla(\epsilon_1)$  and
$V^*=\Delta(-\epsilon_n)=\nabla(-\epsilon_n)$, where $\nabla(\epsilon_1)$ is the co-Weyl module with highest weight $\epsilon_1$. So, both $V$ and $V^*$ are
   tilting $\U_{\kappa}$-module. It is known that the tensor
product of tilting module is again a tilting module~\cite{Dk}. So, $V_\kappa^{r, s}$
is a tilting module for $\U_{\kappa}$.

Over $\mathcal Z$, we have the similar functor $\diamond=\Hom_{\U_\mathcal Z}(-, V^{r, s})$. In
this case,
$$\Delta(\tilde\lambda)^\diamond=\Hom_{\U_\mathcal Z}(\Delta(\tilde\lambda), V^{r, s})_{\mathbb Q(q)}\cap
\Hom_{\mathcal Z}(\Delta(\tilde\lambda), V^{r, s}).$$ By
Corollary~C.19 in \cite{DDPW}, $\Delta(\tilde \lambda)^\diamond$ is
$\mathcal Z$-projective. So, the localization of
$\Delta(\tilde\lambda)^\diamond$ has constant rank for any prime
ideal $\mathfrak p$ of $\mathcal Z$ (See \S7,7 in \cite{Jac}).
Therefore, the dimension of $\Delta(\tilde\lambda)^\diamond$ over
any $\kappa$ is equal to that over $\mathbb Q(q)$.
 Note that  $V^{r,s}$ is  complete reducible
as left $\U_{\mathbb Q(q)}$-module and the multiplicity of $\Delta({\tilde\lambda})$ in $V_{\mathbb Q(q)}^{r,s}$ is
$\dim C(f, \lambda')$~\cite[Lemma~6.5]{KM}, which is  the cardinality of
$\{f_{\lambda, \t, d}\mid (d, \t)\in \mathscr D_{r,s}^f \times
\Std(\lambda')\}$. Therefore, $\{f_{\lambda, \t, d}\mid (d, \t)\in \mathscr D_{r,s}^f \times
\Std(\lambda')\}$  is  a $\kappa$-basis of $\text{Hom}_{\U_\kappa}(\Delta(\tilde \lambda), V_\kappa^{r,s})$.

By definition, $\text{Hom}_{\U_\kappa}(\Delta(\tilde \lambda), V_\kappa^{r,s})$ is the  right $\mathscr B_{r, s, \kappa}$-module
such that $$f.h(x)=f(x)\cdot h, \text{$\forall h\in \mathscr B_{r, s, \kappa}$, $f\in \text{Hom}_{\U_\kappa}(\Delta(\tilde \lambda), V_\kappa^{r,s})$
 and $x\in \Delta ({\tilde\lambda})$.}$$
 Finally,  one can check easily that the linear isomorphism
which sends $f_{\lambda, \t, d}$ to   $e^f\m_{\lambda}g_{w_{\lambda}} \n_{\lambda'} g_{d(\t)} g_{d}  +
\mathscr B_{r, s, \kappa}^{f}$ is a right $\mathscr B_{r,
s, \kappa}$-homomorphism. Now, the result follows from  Lemma~\ref{classcell}.
\end{proof}

The following result gives a classification of singular vectors in $V_\kappa^{r, s}$ with highest weight $\tilde\lambda=\phi(f, \lambda)$ for all $(f, \lambda)\in \Lambda_{r, s}$.
\begin{Theorem}\label{sing} For any  $(f, \lambda)\in \Lambda_{r, s}$ with $0\le f\le \min\{r, s\}$, let $S(\lambda)=\{v_{\lambda, \t, d}|\t\in \Std(\lambda'), d\in \mathscr{D}_{r,s}^f\}$. Then $S(\lambda)$ is a basis of the $\kappa$-space
spanned by all singular vectors of $V_\kappa^{r, s}$ with highest weight $\tilde\lambda$.
\end{Theorem}
\begin{proof} Note that the weight of  each singular vector $v\in V^{r, s}$ is of form $\tilde \lambda$, for some $(f, \lambda)\in \Lambda_{r, s}$, let $M_v$ be the left $\U_{\kappa}$-module
generated by $v$. Then there is a homomorphism $\phi_v$ from Weyl module $\Delta(\tilde \lambda)$ to $M$ sending highest weight vector to highest weight vector.
By Proposition~\ref{weyl-cell}, $\phi_v$ can be written as a $\kappa$-linear combination of $f_{\lambda, \t, d}$'s. Applying these homomorphisms to  $v_\lambda$,
we see that $v$ can be written as a $\kappa$-linear combination of $v_{\lambda, \t, d}$'s, proving the result.
\end{proof}

\section{Decomposition numbers of $\mathscr B_{r,s, \kappa}$ with $\rho^2\in q^{2\mathbb Z}$ }
In this section, we establish explicit
relationship between  decomposition numbers of $\mathscr B_{r, s, \kappa}$ and those for (rational) $q$-Schur algebras.
First, we discuss the case when $e<\infty$. Therefore, $\rho=q^{n+2ke}$ for any $k\in \mathbb Z$.
In this case, we can always assume that $n$ is big enough.
We remark that we use Dipper, Doty and Stoll's presentation for $\mathscr B_{r, s}$ in \cite{DDSII}.
So, we have to use $q^{-1}, \rho^{-1}$ instead of $q, \rho$ respectively, if we use our results in section~2.

\begin{Defn} Suppose  $r,s,n\in \!\!\mathbb Z^{>0}$. Let $\diamond=\text{Hom}_{S_\kappa(n;r,s)} (-, V_\kappa^{r, s})$ and let  $\clubsuit=\text{Hom}_{\mathscr B_{r,s, \kappa}} (-, V_\kappa^{r, s})$,
 where $V_\kappa$ is the $\kappa$-vector space with $\dim_\kappa V_\kappa=n$. \end{Defn}

 For each left $S_\kappa(n;r,s)$-module $M$, let $M^\diamond= \text{Hom}_{S_\kappa(n;r,s)} (M, V_\kappa^{r, s})$.
 Then $M^\diamond$ is a right $\mathscr B_{r, s, \kappa}$-module.
  Similarly, for each right $\mathscr B_{r, s, \kappa}$-module $N$, let
 $N^\clubsuit=\text{Hom}_{\mathscr B_{r,s, \kappa}} (N, V_\kappa^{r, s})$. Then $N^\clubsuit$ is a left $ S_\kappa(n;r,s)$-module.

The following result, which follows from  Proposition~\ref{weyl-cell}, is the key part of our method for determining  the decomposition numbers of
$\mathscr B_{r, s}$ when $\rho^2\in q^{2\mathbb Z}$ and $q$ is a root of unity. Via it, we can set up explicit relationship between indecomposable direct summands of
$V^{rs}$ and principal indecomposable $\mathscr B_{r, s}$-modules.

\begin{Prop}\label{cell-weyl} Suppose $n\ge r+s$. If $\tilde \lambda$ is a highest weight of $V_\kappa^{r, s}$, then $\Delta(\tilde \lambda)^\diamond\cong C(f, \lambda')$ as right $\mathscr B_{r, s, \kappa}$-modules.\end{Prop}
\begin{proof} We consider $V^{r, s}$ over $\mathbb Q(q)$ with generic $q$. It is  complete reducible as left $\U_{\mathbb Q(q)}$-module. Therefore, it can be decomposed into direct
summand of irreducible $\U_{\mathbb Q(q)}$-modules, say $L^\nu$'s  with highest weight $\nu$'s. By Lemma~\ref{phila},  each $\nu$ is of form $\tilde{\lambda}$ for some $(f,\lambda)\in \Lambda_{r, s}$. Further, by Theorem~\ref{sing}, $L^\lambda$ can be generated by certain $v_{\lambda,\t, d}$ for some $\t\in \Std(\lambda')$ and $d\in D^f_{r,s}$. Note that $\U_{\mathbb Q(q)} v_{\lambda,\t, d}$ has $\mathcal  Z$-form
  $\U_{\mathcal Z} v_{\lambda,\t, d}$. Therefore, $\U_{\kappa} v_{\lambda,\t, d}$, which is a left $\U_\kappa$-submodule of $V_\kappa^{r, s}$,
   can be identified with the Weyl module $\Delta(\tilde \lambda)$ of $\U_\kappa$.
So, $\Delta(\tilde\lambda)$ can be considered as a left $S_\kappa(n; r, s)$-module.
Now, the result follows from Proposition~\ref{weyl-cell}.\end{proof}

\begin{Lemma}\label{FGiso} For $r,s,n\in \mathbb Z^{>0}$, let $\mathscr B_{r,s, \kappa}$ and
$S_\kappa(n; r, s)$  be  defined over $\kappa$ with $\rho=q^n$.\begin{enumerate} \item
$\mathscr B_{r, s, \kappa}^\clubsuit\cong V_\kappa^{r, s}$ as left $S_\kappa(n;r,s)$-modules. \item If $n\ge r+s$, then
$(V_\kappa^{r,s})^\diamond\cong \mathscr B_{r, s, \kappa}$ as  right $\mathscr B_{r, s, \kappa}$-modules.\end{enumerate}
\end{Lemma}
\begin{proof} (a) is trivial and (b) follows from Theorem~\ref{ration}(b).\end{proof}
In the remaining part of this section, we keep the assumption that $n\ge r+s$.
Recall that a tilting module for  quantum group  is a module with Weyl filtration and co-Weyl filtration.
Similarly, we have the notion of tilting modules for quasi-hereditary algebras. See, e.g. \cite{Dk}.

 Recall that an
indecomposable tilting module is called a \textsf{partial tilting
module}. By  Theorem~1 in \cite[p208]{Dk}, $V_\kappa^{r, s}$ is a direct
sum of certain partial tilting modules of $S_\kappa(n; r, s)$. Note that any dominant
weight of $V_\kappa^{r, s}$ is of form $\tilde \lambda$ for some $(f,
\lambda)\in \Lambda_{r,s}$. So, any partial tilting module which is
a direct summand of $V_\kappa^{r, s}$  is of form $T(\tilde\lambda)$ with
highest weight $\tilde \lambda$  for some $(f, \lambda)\in
\Lambda_{r,s}$.
Let  $(T(\tilde\lambda) : \Delta (\tilde \mu)) $  be the
multiplicity of  $\Delta( \tilde \mu)$ in  $T(\tilde
\lambda)$.
  It is well known that  $(T(\tilde\lambda) :
\Delta (\tilde\mu)) $   is   independent of   a Weyl
filtration  of $T(\tilde \lambda)$.  We are going to use  $(T(\tilde\lambda) :
\Delta (\tilde\mu)) $'s to determine decomposition numbers of $\mathscr B_{r, s, \kappa}$ over the field $\kappa$.

 Let $S_\kappa(n; r, s)$-mod (resp. $\mathscr B_{r,s, \kappa}$-mod) be the category of left $S_\kappa(n; r, s)$-modules
(resp.$\mathscr B_{r,s, \kappa}$-modules). For each left $S_\kappa(n; r, s)$-module $M$, $\Hom_{S_\kappa(n; r, s)}(V_\kappa^{r,s},M)$
is a left $\mathscr B_{r,s, \kappa}$-module such that,   for any $ x\in V_\kappa^{r,s}$, $b\in\mathscr B_{r,s, \kappa}$ and $\phi\in\Hom_{S_\kappa(n; r, s)}(V_\kappa^{r,s},M)$,
\begin{equation}\label{lbrt} (b\phi)(x)=\phi(xb).\end{equation}
Also,  $V^{r,s}_\kappa\otimes_{\mathscr B_{r,s, \kappa}}N$ is a left  $S_\kappa(n; r, s)$-module for any left $\mathscr B_{r,s, \kappa}$-module $N$.

\begin{Defn}\label{mfg} Let  $\mathbf f$ and $\mathbf g$ be two functors
$$\begin{aligned}\mathbf f: S_\kappa(n; r, s)\text{-mod} &\longrightarrow \mathscr B_{r,s, \kappa}\text{-mod}\\
M &\longmapsto \Hom_{S_\kappa(n; r, s)}(V^{r,s}_\kappa,M)\\
\mathbf g:\mathscr B_{r,s, \kappa}\text{-mod} &\longrightarrow S_\kappa(n; r, s)\text{-mod}\\
N &\longmapsto V^{r,s}_\kappa\otimes_{\mathscr B_{r,s, \kappa}}N
\end{aligned} $$\end{Defn}
Since $\mathbf f$ and $\mathbf g$ are adjoint pairs (see e.g, \cite[Theorem~2.11]{JR}), we have a $\kappa$-linear isomorphism
\begin{equation}\label{ajoint}
\Hom_{S_\kappa(n; r, s)}(\mathbf g(N),M)\cong \Hom_{\mathscr B_{r,s, \kappa}}(N,\mathbf f(M)),
\end{equation}
for any  left $ S_\kappa(n; r, s)$-module $M$ and any left $ \mathscr B_{r,s, \kappa}$-module $N$.

\begin{Lemma}\label{sur}\cite[Theorem~6.11]{DDS1} $S_\kappa(n; r, s)=\varphi(\U_\kappa')$, where $\varphi$ is given in
Theorem~\ref{ration}(a)
and $\U'_\kappa=\U_\kappa(\mathfrak {sl}_n)$. \end{Lemma} 

For any $i, 1\le i\le n$, let  $v_{\hat i}=v_n\otimes \cdots\otimes v_{i+1}\otimes v_{i-1} \otimes
\cdots\otimes v_1\in V_\kappa^{\otimes n-1}$, where \begin{equation}\label{hati} \hat i=(1, 2, \cdots, i-1, i+1, \cdots, n).\end{equation}
The following result has been  given in \cite[Lemma~2.2]{DDS1}.

\begin{Lemma}\label{inject}There is a well defined $\U_\kappa'$-monomorphism $\varphi: V_\kappa^*\longrightarrow V_\kappa^{\otimes n-1}$
such that $\varphi(v_i^*)=(-q)^{i} v_{ \hat{i}} \n_{(n-1)}$ where $\n_{(n-1)}=\sum_{w\in\mathfrak S_{n-1}}(-q)^{\ell(w)}T_w$.
\end{Lemma}
There is a $\U_\kappa'$-monomorphism from $V^{r,s}_\kappa$ to $V_\kappa^{\otimes r+(n-1)s}$ induced by  $\varphi$ in Lemma~\ref{inject}. By abuse of notation,  we denote this monomorphism  by $\varphi$.
Recall that there is an  anti-automorphism $\tau$ of $\U_\kappa$ given by
$$\tau(q^{h_i})=q^{h_i}, \tau(E_i)=F_i, \tau(F_i)=E_i. $$
For any positive integer $m$, Stokke \cite{AS} defined  a  symmetric bilinear  form $(\ ,\ ):V_\kappa^{\otimes m}\times V_\kappa^{\otimes m}\longrightarrow \kappa$ such that
\begin{equation}\label{contra} (v_{\mathbf i},v_{\mathbf j} )=q^{\beta(\mathbf i)}\delta_{\mathbf i,\mathbf j},
\end{equation}
where $\mathbf i, \mathbf j\in I(n,m)$, $\beta(\mathbf i)$ is the number of the pairs $(a,b)$ for which $a<b$ and $i_a\neq i_b$. It is proved in
\cite[Theorem~5.2]{AS} that the bilinear form $(\ , \ )$ in \eqref{contra} is the $\U_{\kappa}$-contravariant form in the sense
\begin{equation} \label{contra12}  (uv,w)=(v,\tau(u)w), u\in\U_{\kappa} ,v,w\in V_\kappa^{\otimes m}.\end{equation}
The following can be considered as a counterpart of the form in \eqref{contra}.

\begin{Defn}\label{mixform} Let  $(\ ,\ ): V^{r,s}_\kappa\times V^{r,s}_\kappa\longrightarrow \kappa$ be the bilinear form such that, for any
$(\mathbf i,\mathbf j), (\mathbf k, \mathbf l) \in I(n,r)\times I^*(n, s)$,
\begin{equation}\label{mf1} (v_{\mathbf i|\mathbf j},v_{\mathbf k|\mathbf  l})=q^{2(j_1+j_2+\cdots+j_s)+\beta(\widehat{\mathbf i|\mathbf j})}\delta_{\mathbf i,\mathbf k}\delta_{\mathbf j, \mathbf l},\end{equation}
where $\widehat{\mathbf i|\mathbf j}=(i_1,i_2,\cdots,i_r, \hat{j_1}, \hat{j_2}, \ldots, \hat{j_s})$ and $\hat j_i$ is defined in \eqref{hati}
\end{Defn}

\begin{Lemma}\label{keycontra}  Let $\phi$ be the  bilinear form on $V^{r,s}_\kappa$ defined  in \eqref{mf1}.
\begin{enumerate} \item $\phi$  is non-degenerate symmetric and  $\U'_\kappa$-contravariant.
\item   $\phi(xb,y)=\phi(x,y\sigma(b))$, for all $ b\in \mathscr B_{r,s, \kappa}$ and $ x,y\in V^{r,s}_\kappa$, where $\sigma$ is the anti-involution defined in Lemma~\ref{inv}.\end{enumerate}
\end{Lemma}

\begin{proof} In fact, the bilinear form $\phi$ is $\U'$-contravariant over $\mathcal  Z$ and hence over $\kappa$. In order to see it, we consider
  the $\U'$-contravariant form on $V^{\otimes r+(n-1)s}$ in ~\eqref{contra} over $\mathbb Q(q)$.
By Lemma~\ref{inject}, there is a $\U'$-monomorphism $\varphi: V^{r,s}\longrightarrow V^{\otimes r+(n-1)s}$ over $\mathbb Q(q)$. So, there is a
$\U'$-contravariant form, say $(\ ,\ )_q$ on $V^{r,s}$, such that
$$(x,y)_q=(\varphi(x),\varphi(y)), x,y\in V^{r,s}. $$
By (\ref{contra}), $(\ ,\ )_q= (\sum_{w\in\mathfrak S_{n-1}}q^{2\ell(w)})^{s} \phi$, forcing   $\phi$ to be $\U'$-contravariant on $V^{r,s}$ over $\mathbb Q(q)$ and hence over $\mathcal Z$. The others  in (a) are clear.

We claim $\phi(v_{\mathbf i|\mathbf j}T_i,v_{\mathbf k|\mathbf l})=\phi(v_{\mathbf i|\mathbf j},v_{\mathbf k|\mathbf l}T_i)$. Without loss of generality, we can assume that  $ \mathbf j=\mathbf l$, $i=1$, and $\mathbf i=(i_1,i_2)$, $i_1<i_2$ and
$ i_1=k_2$, $i_2=k_1$. In this case, we have $\beta(\widehat{\mathbf k|\mathbf l})=\beta(\widehat{\mathbf i|\mathbf j})$. A routine computation verifies our claim. By symmetry, $\phi(v_{\mathbf i|\mathbf j}T^*_j,v_{\mathbf k|\mathbf l})=\phi(v_{\mathbf i|\mathbf j},v_{\mathbf k|\mathbf l}T_j^*)$. Finally, we verify $\phi(v_{\mathbf i|\mathbf j}e_1,v_{\mathbf k|\mathbf l})=\phi (v_{\mathbf i|\mathbf j},v_{\mathbf k|\mathbf l}e_1)$. In this case, we can assume that
$\mathbf i=(i),\mathbf k=(k)$, and $j_1=i$, $l_1=k$, $j_m=l_m$, for $m=2,3,\cdots,s$. So, $\beta(\widehat{\mathbf k|\mathbf l})=\beta(\widehat{\mathbf i|\mathbf j})$ and
$$\phi(v_{\mathbf i|\mathbf j}e_1,v_{\mathbf k|\mathbf l})=q^{2(i+k)-n-1}q^{\beta(\widehat{\mathbf k|\mathbf l})}= q^{2(i+k)-n-1}q^{\beta(\widehat{\mathbf i|\mathbf j})}=\phi (v_{\mathbf i|\mathbf j},v_{\mathbf k|\mathbf l}e_1).$$
This completes the proof of (b).
\end{proof}

By Lemma~\ref{sur}, any left $S_\kappa(n; r, s)$-module can be considered as a left $\U'_\kappa$-module.
For any left $S_\kappa(n; r, s)$-module $N$, let $N^{\circ}$  be the left $S_\kappa(n; r, s)$-module  such that $N^{\circ}=N^*$  as $\kappa$-vector space, and the action is given by
\begin{equation} \label{contramod}
(u\phi)(x)=\phi(\tau(u)x), x\in N, u\in\U_\kappa',\phi\in N^*\end{equation}
For any right $\mathscr B_{r,s, \kappa}$-module $M$, let  $M^{\circ}$ be the right  $\mathscr B_{r,s, \kappa}$-module  such that $M^{\circ}=M^*$ as $\kappa$-vector space, and the action is given by
\begin{equation} \label{contramod1} (\phi b)(y)=\phi(y\sigma(b)), y\in M, b\in\mathscr B_{r,s, \kappa},\phi\in M^*\end{equation}

\begin{Lemma}\label{iso} As  $(S_\kappa(n; r, s),\mathscr B_{r,s, \kappa} )$ bi-modules, $V^{r,s}_\kappa\cong  (V^{r,s}_\kappa)^\circ$. \end{Lemma}
\begin{proof}  The required isomorphism  $\Phi$ follows from  Lemma~\ref{keycontra} if we define
 $\Phi:V^{r,s}_\kappa\longrightarrow (V^{r,s}_\kappa)^\circ$ such that
$\Phi(x)(y)=(x,y)$ for all $x,y\in V^{r,s}_\kappa$, where $(\ ,\ )$ is given in Definition~\ref{mixform}.
\end{proof}

\begin{Lemma}\label{g} Suppose that $T$ is an indecomposable direct summand of  $S_\kappa(n; r, s)$-module $V^{r,s}_\kappa$.  Then $\mathbf g\mathbf f(T)\cong T $.
\end{Lemma}
\begin{proof} Since we are assuming that $r+s\le n$, $ \mathbf f(V^{r,s}_\kappa)\cong \mathscr B_{r,s, \kappa}$ and  $\mathbf  g\mathbf f(V^{r,s}_\kappa)\cong V^{r,s}_\kappa$.
It is easy to see that there is an epimorphism from  $gf(M)$ to $ M$ for any indecomposable direct summand  $M$ of $V^{r,s}_\kappa$ as $S_\kappa(n; r, s)$-modules.
Comparing the dimensions yields the isomorphism as required.
\end{proof}
\begin{Lemma}\label{schur} Any partial tilting module which appears as an indecomposable direct summand of $V^{r,s}_\kappa$ is of
form $T(\tilde{\lambda})$ for some $(f, \lambda')\in \Lambda_{r, s}$ with  $\lambda$ being $e$-regular.  Further we have the following isomorphisms as left $\mathscr B_{r,s, \kappa}$-modules:
\begin{enumerate}
\item for $(\ell, \mu)\in \Lambda_{r, s}$, $\mathbf f(\nabla(\tilde{\mu}))\cong\Hom_{S_\kappa(n; r, s)}(\Delta(\tilde{\mu}),V^{r,s}_\kappa)$;
\item $\mathbf f(T(\tilde{\lambda}))\cong P(f,\lambda')$.
\end{enumerate}
\end{Lemma}

\begin{proof} It follows from Lemma~\ref{iso} that $V^{r,s}_\kappa\cong (V^{r,s}_\kappa)^\circ$.
By \cite[Proposition~4.1.6]{Dk},  $\nabla(\tilde{\mu})\cong \Delta(\tilde{\mu})^\circ$.
In order to prove (a), it suffices to  prove  the following isomorphism as  left $\mathscr B_{r,s, \kappa}$-modules:
 \begin{equation}\label{dual}
 \Hom_{S_\kappa(n; r, s)}(\Delta(\tilde{\mu}),V^{r,s}_\kappa)\cong  \Hom_{S_\kappa(n; r, s)}((V^{r,s}_\kappa)^\circ,\Delta(\tilde{\mu})^\circ).
 \end{equation}
Obviously, $ \Psi:  \Hom_{S_\kappa(n; r, s)}(\Delta(\tilde{\mu}),V^{r,s}_\kappa)\rightarrow  \Hom_{S_\kappa(n; r, s)}((V^{r,s}_\kappa)^\circ,\Delta(\tilde{\mu})^\circ)
$ given by
$$\Psi(\phi)(v^*):x\mapsto v^*(\phi(x)),$$ for any
$ \phi\in \Hom_{S_\kappa(n; r, s)}(\Delta(\tilde{\mu}),V^{r,s}_\kappa), v\in V^{r,s}_\kappa, x\in \Delta(\tilde{\mu})  $
is a $\kappa$-linear isomorphism.  For any  $b\in\mathscr B_{r,s, \kappa}$,
 we have $b\Psi(\phi)(v^*)=\Psi(\phi)(v^*b)$ and
$$\begin{aligned} \Psi(b\phi)(v^*)&:x\mapsto v^*((b\phi)(x))=v^*(\phi(x)\sigma(b)),\cr
\Psi(\phi)(v^*b)&: x \mapsto v^*b(\phi(x))= v^*(\phi(x)\sigma(b)).\cr\end{aligned}$$
So $\Psi(b\phi)=b\Psi(\phi)$, and (a) follows.

(b) By \cite[II,Proposition 2.1(c)]{ARS}, The functor $\mathbf f$ induces a category equivalence between the direct sums of direct summands of
the $ S_\kappa(n; r, s)$-module $V^{r,s}_\kappa$ and the projective $\mathscr B_{r,s, \kappa}$-modules. So $ \mathbf f(T(\tilde{\mu}))$ is an indecomposable projective $\mathscr B_{r,s, \kappa}$-module. For any $(k,\nu')\in\Lambda_{r,s}$
by Lemma~\ref{g}, Proposition~\ref{cell-weyl},  (\ref{ajoint}) and (a), we have   $\kappa$-linear isomorphism
\begin{equation} \label{decom-hom}\begin{aligned} 
& \Hom_{S_\kappa(n; r, s)} (T(\tilde{\mu}),\nabla(\tilde{\nu}))\cong \Hom_{S_\kappa(n; r, s)} (\mathbf g\mathbf f(T(\tilde{\mu})),\nabla(\tilde{\nu}))\\
\cong & \Hom_{S_\kappa(n; r, s)} (\mathbf f (T(\tilde{\mu})) , \mathbf f(\nabla(\tilde{\nu})))\cong
\Hom_{\mathscr B_{r,s, \kappa}}( P(f,\lambda'),C(k,\nu'))
\\ \end{aligned}
\end{equation}
for some $(f, \lambda')\in \Lambda_{r, s}$ such that $\mathbf f(T(\tilde \mu))=P(f,\lambda')$ with $\lambda$ being $e$-regular.
We remark that $C(k, \nu')$ is considered as a right $\mathscr B_{r, s, \kappa}$-module in  Proposition~\ref{cell-weyl}. Using anti-involution $\sigma$ in Lemma~\ref{inv}, it can be considered as the 
left $\mathscr B_{r, s, \kappa}$-module in \eqref{decom-hom}.
By \cite[Lemma~2.18]{Ma},
\begin{equation}\label {decom-hom1}   \dim_\kappa\Hom_{\mathscr B_{r,s}}( P(f,\lambda'),C(k,\nu'))=[C(k,\nu'): D^{f, \lambda'}].\end{equation}
 We have  $(\ell,\mu')\unrhd (f,\lambda')$ by assuming $\mu=\nu$. If $\nu=\lambda$, then  $$\Hom_{S_\kappa(n; r, s)}(T(\tilde{\mu}),\nabla(\tilde{\lambda}))\neq 0,$$
forcing $\tilde \lambda\unlhd \tilde \mu$. So,   $(\ell,\mu')\unlhd (f,\lambda')$, $f=\ell$ and $\mu=\lambda$.  This proves (b).\end{proof}

\begin{Theorem}\label{compoeq} Suppose $(f, \lambda'), (\ell, \mu')\in \Lambda_{r, s}$
such that $\lambda$ is $e$-regular.
 Then
$$(T(\tilde \lambda): \Delta({\tilde \mu}))=[C(\ell, \mu'): D^{f, \lambda'}].$$  \end{Theorem}
\begin{proof} Since $(T(\tilde \lambda): \Delta({\tilde \mu}))=\dim_\kappa \Hom_{S_\kappa(n; r, s)}(T(\tilde{\mu}),\nabla(\tilde{\nu}))$,  the result follows from \eqref{decom-hom}--\eqref{decom-hom1}
\end{proof}

In the remaining part of this section, we consider right $\mathscr B_{r, s, \kappa}$-modules. 
Of course, Theorem~\ref{compoeq} can be read  for right  $\mathscr B_{r, s, \kappa}$-modules since 
any left  $\mathscr B_{r, s, \kappa}$-module can be considered as a right  $\mathscr B_{r, s, \kappa}$-module
via the anti-involution $\sigma$ in Lemma~\ref{inv}.

\begin{Theorem}\label{block2} Suppose $e<\infty$.  Let $\mathscr
B_{r,s, \kappa}$ be defined over $\kappa$ with defining parameter
$\rho=q^n$. If  $(f, \lambda'), (\ell, \mu')\in \Lambda_{r, s}$,
then $C(f, \lambda')$ and $C(\ell, \mu')$ are in the same $\mathscr
B_{r, s, \kappa}$ block if and only if $\Delta({\tilde\lambda})$ and
$\Delta({\tilde \mu})$ are in the same $S_\kappa(n; r, s)$-block.
\end{Theorem}

\begin{proof} Suppose that  $C(f, \lambda')$ and $C(\ell, \mu')$ are in the same
$\mathscr B_{r, s, \kappa}$-block. Without loss of any generality, we assume
that $ [C(f, \lambda'): D^{(\ell, \mu')}]\neq 0$. So, $\mu$ is
$e$-regular.  By Theorem~\ref{compoeq},
$(T(\tilde\mu):\Delta({\tilde\lambda}))\neq 0$. Since
$\Delta({\tilde\mu})$ is the unique bottom  section of any Weyl
filtration of the partial tilting module  $T(\tilde \mu)$,
$\Delta({\tilde \lambda})$ and $\Delta({\tilde \mu})$ are in the
same $S_\kappa(n; r, s)$-block.

Conversely, let $Y$ be a $\mathscr B_{r,s, \kappa}$-module which is an
indecomposable direct summand of $V_\kappa^{r,s}$.  Then $  Y^\clubsuit\neq 0$. By definition of $S_\kappa(n; r, s)$, $ Y^\clubsuit$  is a direct summand of
$S_\kappa(n; r,s)$.

 We claim that $Y^\clubsuit $ is indecomposable.
Otherwise,  $Y^\clubsuit $ is a direct
sum of certain principal indecomposable $S_\kappa(n;r,s)$-modules, say
$P(\tilde\lambda)$'s. By Proposition~\ref{weyl-cell},
$0\neq P(\tilde\lambda)^\diamond$, which is a direct
summand of $\mathscr B_{r,s, \kappa}$-module $V_\kappa^{r,s}$. Counting the number
of indecomposable direct summands of $\mathscr B_{r,s, \kappa}$-module
$V_\kappa^{r,s}$  gives a contradiction. So, both  $Y^\clubsuit $ and  $P(\tilde\lambda)^\diamond$ are
indecomposable and hence $Y^\clubsuit =P(\tilde\lambda)$ for some
$\tilde\lambda\in \Lambda^+(r,s)$.

Suppose $\Delta({\tilde\lambda})$ and $\Delta({\tilde \mu})$ are in
the same $S_\kappa(n; r, s)$-block. Without loss of any generality,
we can assume $(P(\tilde\lambda): \Delta ({\tilde \mu}))\neq 0$.
Applying $\diamond $  to $P(\tilde\lambda)$, we see that both
$C(f, \lambda')$ and $C(\ell, \mu')$ appear as sections of a cell
filtration of the indecomposable right $\mathscr B_{r,s, \kappa}$-module
$P(\tilde\lambda)^\diamond$.
 So, $C(f, \lambda')$ and $C(\ell,
\mu')$ are in the same block.
\end{proof}

By Theorem~\ref{block2}, we know that  blocks of
$\mathscr B_{r,s, \kappa}$ can be determined by those of rational $q$-Schur algebras.
We remark that we will study  blocks of $\mathscr B_{r,s, \kappa}$ in details elsewhere.

Each  indecomposable   direct summand of right $\mathscr
B_{r,s, \kappa}$-module $V_\kappa^{r, s}$ will be called a  Young module.
 Let $Y({f, \lambda'})=P(\tilde\lambda)^\diamond$. Using standard
arguments on tilting module $V_\kappa^{r,s}$, we have the following result
immediately.

\begin{Cor}\label{blockeq} Suppose $e<\infty$.  Let $\mathscr B_{r,s, \kappa}$ be defined over $\kappa$ with $\rho=q^n$ for $n\gg0$.
Suppose $(f, \lambda'), (\ell, \mu')\in \Lambda_{r, s}$. Then
$Y({f, \lambda'})$ has a filtration of right cell modules of
$\mathscr B_{r,s, \kappa}$ with bottom section $C(f, \lambda')$. Further, the multiplicity of $ C(\ell, \mu')$ in the previous 
filtration of 
$Y({f, \lambda'})$ is $(P(\tilde\lambda):
\Delta({\tilde\mu}))=[\Delta(\tilde \mu): L^{\tilde \lambda}], $ where $L^{\tilde \lambda}$ is the irreducible $S_\kappa(n; r, s)$-module with highest weight $\tilde \lambda$.
\end{Cor}

The following result is motivated by Donkin and Tange's work in \cite{DT}.


\begin{Theorem} \label{einfty}If  $\rho^2= q^{2a}$ for some  $a\in \mathbb Z$ with $|a|\le r+s-2$ and $e=\infty$, then $[C(f, \lambda'): D^{\ell, \mu'}]$
is the same as that under the assumption $e\gg0$.
\end{Theorem}

\begin{proof}
 Let $t$ be an indeterminate.  We consider $\kappa [t, t^{-1}]$ which is
 the localization of $\kappa [t]$ at $t$. So, $\kappa [t,
 t^{-1}]$ is a Dedeking ring with quotient field $K=\kappa (t)$.
Since $\mathscr B_{r,s}$ is cellular over $\mathbb Z[q, q^{-1},
\rho, \rho^{-1}, (q-q^{-1})^{-1}]$, we see that $\mathscr B_{r,s}$
is free over $\kappa [t, t^{-1}]$ with defining parameters $\rho$
and $t$ such that $\rho^2=t^{2a}$. Further, it is a
$\kappa [t, t^{-1}]$-lattice of $\mathscr B_{r,s, K}$  in the
sense of \cite[5.2]{DT}. Let $\epsilon^2 $ be primitive $k$-th root of unity
in $\kappa $ and  let $M_\epsilon\subset \kappa [t, t^{-1}]$ be
the maximal ideal generated by $t-\epsilon$. Then $\kappa =\kappa [t, t^{-1}]/M_\epsilon$. In this case, we use $\kappa_\epsilon$ instead of $\kappa$ so as  to emphasis $\epsilon$.

Let $Grot (\mathscr B_{r,s, F})$ be the Grothendieck group of finite
dimension $\mathscr B_{r,s, F}$-modules over $F$, a  field  which is a $\kappa [t, t^{-1}]$-algebra. By
\cite[5.2(1)]{DT}, $Grot (\mathscr B_{r,s, K})\cong Grot (\mathscr B_{r,s,\kappa_\epsilon})$ if $k$, the order of $\epsilon^2$,
is big enough. Since $\mathscr B_{r,s}$ is a cellular algebra,  any
cell module $C(f, \lambda)$ of $\mathscr B_{r,s}$   can be considered as  $\kappa [t, t^{-1}]$-lattice
of the corresponding cell module $C(f, \lambda)_K$ of $\mathscr
B_{r,s, K}$. Therefore, the decomposition matrices of
$\mathscr B_{r,s}$ over $K$ and $\kappa_\epsilon$ are the same if the order of $\epsilon^2$ is big enough.
\end{proof}

Finally, we explain why  decomposition numbers of $\mathscr B_{r, s, \kappa}$ can be computed via those for
$q$-Schur algebras if $\rho^2\in q^{2\mathbb Z}$.

Suppose  $(f, \lambda')\in \Lambda_{r, s}$.  We have
$\U_\kappa$-module $ \Delta({\tilde \lambda})$.
When $\lambda$ is $e$-regular, we have $T(\tilde \lambda)$, an
indecomposable  direct summand of $V_\kappa^{r,s}$ with respect to the
highest weight $\tilde\lambda$. Both $T(\tilde\lambda)$ and
$\Delta(\tilde\lambda)$ are  rational representations of
$\U_\kappa$. When we consider the restriction of
such  modules to $\U_\kappa (\mathfrak {sl}_n)$, they are isomorphic
to $T({\tilde \lambda+s\omega}), \Delta ({\tilde \lambda+s\omega})$
with $\omega=(1, \cdots, 1)\in \Lambda^+(n)$, the corresponding
polynomial representations of $\U_\kappa $.
Further, we have the following well-known equalities:
\begin{equation}\label{rapo} (T({\tilde \lambda}): \Delta({\tilde \mu}))
=(T({\tilde \lambda+s\omega}): \Delta({\tilde
\mu+s\omega}))=[ \Delta(\alpha): L^\beta]\end{equation} 
where $\alpha$ (resp. $\beta$) is the conjugate of 
${\tilde \mu+s\omega}$ (resp. ${\tilde \lambda+s\omega}$).
We remark that the last equality follows from \cite[Proposition~4.1e]{Dk}.
 So, the decomposition numbers for $\mathscr B_{r, s, \kappa}$ can be computed via those for $q$-Schur algebras if $q^2\in q^{2\mathbb Z}$.
Finally, if  $\rho^2\not\in q^{2\mathbb Z}$, by  Theorems~\ref{blocks1},  decomposition numbers
of $\mathscr B_{r, s, \kappa}$ can be computed by those for Hecke algebras  associated to symmetric groups.

If the ground field $\kappa$ is $\mathbb C$, we can use Ariki's result in \cite{Ar} and
Varagnolo and Vasserot's results in \cite{VV}. In the latter case, we have to use $q$ instead of $q^{-1}$.
 In \cite{VV}, Varagnolo and Vasserot used  $\U_q(\mathfrak {gl}_n)$ in \cite{Jan}.
   By \cite[Remark~1.25]{APW}, we  need to use   $w_0(s\omega+\tilde \lambda)$ instead of our $s\omega+\tilde \lambda$ when we use corresponding result in \cite{VV}, where $w_0$ is the longest element in $\mathfrak S_n$.
  In summary, when the ground field is $\mathbb C$, decomposition numbers of $\mathscr B_{r, s, \kappa}$ can be computed via the values of inverse Kazhdan-Lusztig polynomials at $q=1$ associated to certain extended affine Weyl groups of type $A$. We leave the details to the reader.

 Cox and De~Visscher~\cite{CV} proved that  decomposition numbers of walled Brauer algebras over $\mathbb C$ are either $0$ or $1$. This should correspond to our result for
 $\mathscr B_{r, s}$ over $\mathbb C$ with $o(q)=\infty$ and $\rho^2\in q^{2\mathbb Z}$. Finally, it is natural to ask whether
 one can find  results for quantum general linear superalgebras and quantized walled Brauer algebras similar to those for general linear Lie superalgebras
 and walled Brauer algebras in
  \cite{BS}.

\providecommand{\bysame}{\leavevmode ---\ } \providecommand{\og}{``}
\providecommand{\fg}{''} \providecommand{\smfandname}{and}
\providecommand{\smfedsname}{\'eds.}
\providecommand{\smfedname}{\'ed.}
\providecommand{\smfmastersthesisname}{M\'emoire}
\providecommand{\smfphdthesisname}{Th\`ese}

\end{document}